\pdfoutput=1
\RequirePackage{ifpdf}
\ifpdf 
\documentclass[pdftex]{sigma}
\else
\documentclass{sigma}
\fi

\usepackage{euscript}
\usepackage{enumerate}

\usepackage[all]{xy}

\newcommand{\RN}{\mathbb{R}} 
\newcommand{\CN}{\mathbb{C}} 
\newcommand{\ZN}{\mathbb{Z}} 

\newcommand{\eps}{\ensuremath\varepsilon}

\newcommand{\suml}{\sum\limits}

\newcommand{\ovl}{\overline}

\newcommand{\gr}{\mathop{\mathrm{gr}}\nolimits}

\renewcommand{\Im}{\mathop{\mathrm{Im}}\nolimits}

\newcommand{\SU}{\mathop{\mathrm{SU}}\nolimits}

\newcommand{\ad}{\ensuremath\operatorname{ad}}

\newcommand{\mc}[1]{\mathcal{#1}}

\renewcommand{\Re}{\operatorname{Re}}

\newcommand{\C}{\mathbb{C}}

\newcommand{\g}{\mathfrak{g}}
\newcommand{\bP}{\mathbf{P}}

\newtheorem{Theorem}{Theorem}[section]
\newtheorem{Lemma}[Theorem]{Lemma}
\newtheorem{Proposition}[Theorem]{Proposition}
\newtheorem{Corollary}[Theorem]{Corollary}
\theoremstyle{definition}

\newtheorem{Example}[Theorem]{Example}
\newtheorem{Remark}[Theorem]{Remark}
\numberwithin{equation}{section}

\begin{document}

\newcommand{\arXivNumber}{2009.09437}

\renewcommand{\thefootnote}{}

\renewcommand{\PaperNumber}{029}

\FirstPageHeading

\ShortArticleName{Twisted Traces and Positive Forms on Quantized Kleinian Singularities of Type A}

\ArticleName{Twisted Traces and Positive Forms\\ on Quantized Kleinian Singularities of Type A\footnote{This paper is a~contribution to the Special Issue on Representation Theory and Integrable Systems in honor of Vitaly Tarasov on the 60th birthday and Alexander Varchenko on the 70th birthday. The full collection is available at \href{https://www.emis.de/journals/SIGMA/Tarasov-Varchenko.html}{https://www.emis.de/journals/SIGMA/Tarasov-Varchenko.html}}}

\Author{Pavel ETINGOF~$^{\rm a}$, Daniil KLYUEV~$^{\rm a}$, Eric RAINS~$^{\rm b}$ and Douglas STRYKER~$^{\rm a}$}

\AuthorNameForHeading{P.~Etingof, D.~Klyuev, E.~Rains and D.~Stryker}

\Address{$^{\rm a)}$~Department of Mathematics, Massachusetts Institute of Technology, USA}
\EmailD{\href{mailto:etingof@math.mit.edu}{etingof@math.mit.edu}, \href{mailto:klyuev@mit.edu}{klyuev@mit.edu}, \href{mailto:stryker@mit.edu}{stryker@mit.edu}}

\Address{$^{\rm b)}$~Department of Mathematics, California Institute of Technology, Pasadena, CA 91125, USA}
\EmailD{\href{mailto:rains@mit.edu}{rains@mit.edu}}

\ArticleDates{Received September 22, 2020, in final form March 08, 2021; Published online March 25, 2021}

\Abstract{Following~[Beem C., Peelaers W., Rastelli L., \textit{Comm. Math. Phys.} \textbf{354} (2017), 345--392] and [Etingof P., Stryker D., \textit{SIGMA} \textbf{16} (2020), 014, 28~pages], we undertake a~detailed study of twisted traces on quantizations of Kleinian singularities of type $A_{n-1}$. In particular, we give explicit integral formulas for these traces and use them to determine when a trace defines a positive Hermitian form on the corresponding algebra. This leads to a classification of unitary short star-products for~such quantizations, a problem posed by Beem, Peelaers and Rastelli in connection with 3-dimensional superconformal field theory. In particular, we confirm their conjecture that for $n\le 4$ a unitary short star-product is unique and compute its parameter as a function of the quantization parameters, giving exact formulas for the numerical functions by Beem, Peelaers and Rastelli. If $n=2$, this, in~particular, recovers the theory of unitary spherical Harish-Chandra bimodules for ${\mathfrak{sl}}_2$. Thus the results of this paper may be viewed as a starting point for a generalization of the theory of~unitary Harish-Chandra bimodules over enveloping algebras of reductive Lie algebras [Vogan~Jr.~D.A., \textit{Annals of Mathematics Studies}, Vol.~118, Princeton University Press, Princeton, NJ, 1987] to~more general quantum algebras. Finally, we derive recurrences to compute the coefficients of short star-products corresponding to twisted traces, which are generalizations of discrete Painlev\'e systems.}


\Classification{16W70; 33C47}

\begin{flushright}
\begin{minipage}{75mm}
\it To Vitaly Tarasov and Alexander Varchenko\\
 with admiration
\end{minipage}
\end{flushright}

\renewcommand{\thefootnote}{\arabic{footnote}}
\setcounter{footnote}{0}


\section{Introduction}

The notion of a {\it short star-product} for a filtered quantization $\mathcal A$ of a hyperK\"ahler cone was introduced by Beem, Peelaers and Rastelli in~\cite{BPR} motivated by the needs of 3-dimensional super\-conformal field theory (under the name ``star-product satisfying the truncation condition"); this is an algebraic incarnation of non-holomorphic $\SU(2)$-symmetry of such cones. Roughly speaking, these are star-products which have fewer terms than expected (in fact, as few as possible). The most important short star-products are {\it nondegenerate} ones, i.e., those for which the constant term ${\rm CT}(a*b)$ of $a*b$ defines a nondegenerate pairing on $A=\operatorname{gr} \mathcal A$. Moreover, physically the most interesting ones among them are those for which an appropriate Hermitian version of this pairing is positive definite; such star-products are called {\it unitary}. Namely, short star-products arising in 3-dimensional SCFT happen to be unitary, which is a motivation to take a closer look at them.

In fact, in order to compute the parameters of short star-products arising from 3-dimensional SCFT, in~\cite{BPR} the authors attempted to classify unitary short star-products for even quantizations of Kleinian singularities of type $A_{n-1}$ for $n\le 4$. Their low-degree computations suggested that in these cases a unitary short star-product should be unique for each quantization. While the~$A_1$ case is easy (as it reduces to the representation theory of ${\rm SL}_2$),
in the $A_2$ case the situation is already quite interesting. Namely, in this case an even quantization depends on one parameter~$\kappa$, and
Beem, Peelaers and Rastelli showed that (at least in low degrees) short star-products for such a quantization are parametrized by another parameter $\alpha$. Moreover, they computed numerically the function $\alpha(\kappa)$ expressing the parameter of the unique {\it unitary} short star-product on the parameter of quantization \cite[Fig.~2]{BPR}, but a formula for this function (even conjectural) remained unknown.

These results were improved upon by Dedushenko, Pufu and Yacoby in~\cite{DPY}, who computed the short star-products coming from 3-dimensional SCFT in a different way. This made the need to understand all nondegenerate short star-products and in particular unitary ones less pressing for physics, but it remained a very interesting mathematical problem.

Motivated by~\cite{BPR}, the first and the last author studied this problem in~\cite{ES}. There they developed a mathematical theory of nondegenerate short star-products and obtained their classification. As a result, they confirmed the conjecture of~\cite{BPR} that such star-products exist for a~wide class of hyperK\"ahler cones and are parametrized by finitely many parameters. The main tool in this paper is the observation, due to Kontsevich, that nondegenerate short star-products correspond to nondegenerate twisted traces on the quantized algebra $\mathcal A$, up to scaling. The reason this idea is effective is that traces are much more familiar objects (representing classes in the zeroth Hochschild homology of $\mathcal{A}$), and can be treated by standard techniques of representation theory and noncommutative geometry. However, the specific example of type $A_{n-1}$ Kleinian singularities and in particular the classification of unitary short star-products was not addressed in detail in~\cite{ES}.

\looseness=1
The goal of the present paper is to apply the results of~\cite{ES} to this example, improving on the results of~\cite{BPR}. Namely, we give an explicit classification of nondegenerate short star-products for type $A_{n-1}$ Kleinian singularities, expressing the corresponding traces of weight $0$ elements (i.e.,~polynomials $P(z)$ in the weight zero generator $z$) as integrals $\int_{{\rm i}\mathbb R} P(x)w(x)|{\rm d}x|$ of $P(x)$ against a certain weight function $w(x)$. As a result, the corresponding quantization map sends monomials $z^k$ to $p_k(z)$, where $p_k(x)$ are monic orthogonal polynomials with weight $w(x)$ which belong to the class of {\it semiclassical orthogonal polynomials}. If $n=1$, or $n=2$ with special parameters, they reduce to classical hypergeometric orthogonal polynomials, but in general they do not. We~also determine which of these short star-products are unitary, confirming the conjecture of~\cite{BPR} that for even quantizations of $A_{n-1}$, $n\le 4$ a unitary star product is unique. Moreover, we find the exact formula for the function $\alpha(\kappa)$ whose graph is given in Fig.~2 of~\cite{BPR}:
\[
\alpha(\kappa)=\frac{1}{4}-\frac{\kappa+\frac{1}{4}}{1-\cos\big(\pi\sqrt{\kappa+\frac{1}{4}}\big)}.
\]
In particular, this recovers the value $\alpha\big({-}\frac{1}{4}\big)=\frac{1}{4}-\frac{2}{\pi^2}$ predicted in~\cite{BPR} and confirmed \mbox{in~\cite{DFPY,DPY}}.

It would be very interesting to develop a similar theory of positive traces for higher-dimen\-sio\-nal quantizations, based on the algebraic results of~\cite{ES}. It would also be interesting to extend this analysis from the algebra $\mc A$ to bimodules over $\mc A$ (e.g., Harish-Chandra bimodules, cf.~\cite{L}). Finally, it would be interesting to develop a $q$-analogue of this theory.
These topics are beyond the scope of this paper, however, and are subject of future research. For~instance, the~$q$-analogue of our results for Kleinian singularities of type A will be worked out by the second author in~a~forthcoming paper~\cite{K2}.

\begin{Remark} We show in Example~\ref{neq2} that for $n=2$ the theory of positive traces developed here recovers the classification of irreducible unitary spherical representations of ${\rm SL}_2(\mathbb C)$~\cite{V}. Moreover,
this can be extended to the non-spherical case if we consider
traces on Harish-Chandra bimodules over quantizations (with different parameters on the left and the right, in general) rather than just quantizations themselves. One could expect that a similar theory for higher-dimensional quantizations, in the special case of quotients of $U(\g)$ by a central character (i.e., quantizations of the nilpotent cone) would recover the theory of unitary representations of the complex reductive group $G$ with Lie algebra $\g$. This suggests that the theory of positive traces on filtered quantizations of hyperK\"ahler cones may be viewed as a generalization of the theory of unitary Harish-Chandra bimodules for simple Lie algebras. A peculiar but essential new feature of this generalization (which may scare away classical representation theorists), is that a given simple bimodule may have more than one Hermitian (and even more than one unitary) structure up to scaling (namely, unitary structures form a cone, often of dimension $>1$), and that a bimodule which admits a unitary structure need not be semisimple.
\end{Remark}

\begin{Remark}
The second author studied the existence of unitary star-products for type $A_{n-1}$ Kleinian singularities in~\cite{K1} and obtained a partial classification of quantizations that admit a~unitary star-product. That paper also contains examples of non-semisimple unitarizable bimo\-dules. The present paper has stronger results: it contains a complete description of the set of~unitary star-products for any type $A_{n-1}$ Kleinian singularity.
\end{Remark}

The organization of the paper is as follows. Section~\ref{sec2} is dedicated to outlining the algebraic theory of filtered quantizations and twisted traces for Kleinian singularities of type A, following~\cite{ES}. In Section~\ref{sec3} we introduce our main analytic tools, representing twisted traces by contour integrals against a weight function. In this section we also use this weight function to study the orthogonal polynomials arising from twisted traces. In~Section~\ref{sec4}, using the analytic approach of Section~\ref{sec3}, we determine which twisted traces are positive. In~particular, we confirm the conjecture of~\cite{BPR} that a positive trace is unique up to scaling for $n\le 4$ (for the choice of conjugation as in~\cite{BPR}), and find the exact dependence of the parameter of the positive trace on the quantization parameters for $n=3$ and $n=4$, which was computed numerically in~\cite{BPR}.\footnote{It is curious that, unlike classical representation theory, this dependence is given by a transcendental function.}
Finally, in Section~\ref{expcom} we discuss the problem of explicit computation of the coefficients $a_k$, $b_k$
of the 3-term recurrence for the orthogonal polynomials arising from twisted traces, which
appear as coefficients of the corresponding short star-product. Since these orthogonal polynomials are semiclassical, these coefficients can be computed using non-linear recurrences which are generalizations of discrete Painlev\'e systems.

\section{Filtered quantizations and twisted traces}\label{sec2}

\subsection{Filtered quantizations}
Let $X_n$ be the Kleinian singularity of type $A_{n-1}$.
Recall that
\[
A:=\mathbb C[X_n]=\CN[p,q]^{\mathbb Z/n},
\]
where
$\mathbb Z/n$ acts by $p\mapsto {\rm e}^{2\pi {\rm i}/n}p,\ q\mapsto {\rm e}^{-2\pi {\rm i}/n}q$.
Thus
\[
A=\mathbb C[u,v,z]/(uv-z^n),
\]
where
\[
u=p^n,\qquad
v=q^n,\qquad
z=pq.
\]
This algebra has a grading defined by the formulas $\deg(p)=\deg(q)=1$, thus
\begin{equation}\label{gra}
 \deg(u)=\deg(v)=n,\qquad
 \deg(z)=2.
 \end{equation}
The Poisson bracket is given by $\lbrace{p,q\rbrace}=\frac{1}{n}$
and on $A$ takes the form
\[
 \lbrace z,u\rbrace =-u,\qquad
 \lbrace z,v\rbrace=v,\qquad
 \lbrace u,v\rbrace=nz^{n-1}.
\]
Also recall that filtered quantizations $\mathcal A$ of $A$ are {\it generalized Weyl algebras}~\cite{B} which look as follows. Let~$P\in \mathbb C[x]$ be a monic polynomial of degree $n$. Then $\mc{A}=\mc A_P$ is the algebra generated by $u$, $v$, $z$ with defining relations
\[
[z,u]=-u,\qquad
[z,v]=v,\qquad
vu=P\big(z-\tfrac{1}{2}\big),\qquad
uv=P\big(z+\tfrac{1}{2}\big)
\]
and filtration defined by \eqref{gra}.
Thus we have
\[
[u,v]=P\big(z+\tfrac{1}{2}\big)-P\big(z-\tfrac{1}{2}\big)=nz^{n-1}+\cdots ,
\]
i.e., the quasiclassical limit indeed recovers the algebra $A$ with the above Poisson bracket.

Note that we may consider the algebra $\mc A_P$ for a polynomial $P$ that is not necessarily monic. However, we can always reduce to the monic case by rescaling $u$ and/or $v$. Also by transformations $z\mapsto z+\beta$ we can make sure that the subleading term of $P$ is zero, i.e.,
\[
P(x)=x^n+c_2x^{n-2}+\dots +c_n.
\]
Thus the quantization~$\mc A$ depends on $n-1$ essential parameters (the roots of $P$, which add up to zero).

The algebra $\mc A$ decomposes as a direct sum of eigenspaces of
$\ad z$:
\[
\mc{A}=\oplus_{k\in \ZN}\mc{A}_{k}.
\]
If $b\in \mc A_m$, we will say that $b$ has {\it weight} $m$.
The weight decomposition of $\mc A$ can be viewed as a~$\mathbb C^\times$-action; namely, for $t\in \mathbb C^\times$ let
\begin{equation}\label{gt}
g_t=t^{\ad z}\colon \ \mc{A}\to \mc{A}
\end{equation}
be the automorphism of $\mc{A}$ given by
\[
g_t(v)= tv,\qquad
g_t(u)=t^{-1}u,\qquad
g_t(z)=z.
\]
Then $g_t(b)=t^mb$ if $b$ has weight $m$.

\begin{Example} \qquad
\begin{enumerate}\itemsep=0pt
\item[$1.$] Let $n=1$, $P(x)=x$. Then $\mc A$ is the Weyl algebra
generated by $u$, $v$ with $[u,v]=1$, and $z=vu+\tfrac{1}{2}=uv-\tfrac{1}{2}$.

\item[$2.$] Let $n=2$ and $P(x)=x^2-C$. Then setting $e=v$, $f=-u$, $h=2z$, we get
\[
[h,e]=2e,\qquad
[h,f]=-2f,\qquad
[e,f]=h,\qquad
fe=-\big(\tfrac{h+1}{2}\big)^2+C,
\]
i.e., $\mc A$ is the quotient of the universal enveloping algebra $U(\mathfrak{sl}_2)$ by the relation $fe+\big(\tfrac{h+1}{2}\big)^2=C$, where $fe+\big(\tfrac{h+1}{2}\big)^2$ is the Casimir element.
\end{enumerate}
\end{Example}

\subsection{Even quantizations}
Let $s$ be the automorphism of $\mc A$
given by
\[
s(u)=(-1)^nu,\qquad s(v)=(-1)^nv,\qquad s(z)=z,
\]
in other words, we have $s=g_{(-1)^n}$. Thus $\gr s\colon A\to A$
equals $(-1)^d$, where $d$ is the degree operator.
Recall \cite[Section~2.3]{ES}, that a filtered quantization $\mathcal{A}$ is called {\it even} if it is equipped with an antiautomorphism $\sigma$ such that $\sigma^2=s$ and $\gr\sigma={\rm i}^d$, and that $\sigma$ is unique if exists~\cite[Remark~2.10]{ES}. This means that $\sigma (z)=-z$, $\sigma (u)={\rm i}^n u$, $\sigma (v)={\rm i}^n v$. It is easy to see that $\sigma$ exists if and only if
\[
(-1)^n P\big(z-\tfrac{1}{2}\big)=(-1)^n vu=\sigma(v)\sigma(u)=\sigma (uv)=\sigma\big(P\big(z+\tfrac{1}{2}\big)\big)=P\big({-}z+\tfrac{1}{2}\big).
\]
This is equivalent to
\[
P(-x)=(-1)^n P(x),
\]
i.e., $P$ contains only terms $x^{n-2i}$. Thus even quantizations of $A$ are parametrized by $[n/2]$ essential parameters, and all quantizations for $n\le 2$ are even.

\subsection{Quantizations with a conjugation and a quaternionic structure}\label{conju}

Recall \cite[Section~3.6]{ES} that a conjugation on~$\mc A$ is an antilinear filtration preserving automorphism $\rho\colon \mc A\to \mc A$ that commutes with $s$. We~will consider conjugations on~$\mc A$
given by
\begin{equation}\label{rho}
\rho(v)=\lambda u,\qquad \rho(u)=\lambda_* v,\qquad \rho(z)=-z,
\end{equation}
where $\lambda,\lambda_*\in \mathbb C^\times$; it easy to show that they are the only ones up to symmetry, using that any two such conjugations differ by a filtration preserving automorphism commuting with $s$. The~auto\-morphism $u\mapsto \gamma^{-1} u$, $v\mapsto \gamma v$ rescales $\lambda$ by $|\gamma|^{-2}$ and $\lambda_*$ by $|\gamma|^2$, so we may assume that $|\lambda|=1$, i.e.,
\[
\lambda=\pm {\rm i}^{-n}{\rm e}^{-\pi {\rm i} c},
\]
where $c\in [0,1)$.
Then
\[
\overline{P}\big({-}z+\tfrac{1}{2}\big)=\rho\big(P\big(z+\tfrac{1}{2}\big)\big)
=\rho(uv)=\rho(u)\rho(v)=\lambda\lambda_* vu=\lambda\lambda_* P\big(z-\tfrac{1}{2}\big),
\]
i.e., $\overline{P}(-x)=\lambda\lambda_* P(x)$. Thus $\lambda_*=(-1)^n\lambda^{-1}=\pm {\rm i}^{-n}{\rm e}^{\pi {\rm i}c}$ (so $|\lambda_*|=1$) and
\[
\ovl{P}(-x)=(-1)^n P(x),
\]
i.e., ${\rm i}^nP$ is real on ${\rm i}\mathbb R$. We~also have
\[
\rho^2(u)=\overline{\lambda_*}\lambda u,\qquad \rho^2(v)=\overline{\lambda}\lambda_* v,\qquad \rho^2(z)=z,
\]
so $\rho^2=g_t$, where $g_t$ is given by \eqref{gt} and
\[
t=(-1)^n\overline{\lambda}\lambda^{-1}=(-1)^n\lambda^{-2}={\rm e}^{2\pi {\rm i}c},
\]
i.e., $|t|=1$. Thus we see that for every $t$ there are two non-equivalent conjugations, corresponding to the two choices of sign for $\lambda$, which we denote by $\rho_+$ and $\rho_-$.

In particular, consider the special case $t=(-1)^n$, i.e., $g_t=s$. Then $c=\frac{1}{2}$ for $n$ odd and $c=0$ for $n$ even.
Thus $\lambda=\pm 1$, so the conjugation $\rho$ on~$\mc A$ is given by
\[
\rho(v)=\pm u,\qquad \rho(u)=\pm (-1)^nv,\qquad \rho(z)=-z.
\]

Now assume in addition that $\mc A$ is even, i.e., $P(-x)=(-1)^nP(x)$. Then we have $\rho\sigma=\sigma^{-1}\rho$, so $\rho$ and $\sigma$ give a {\it quaternionic structure} on $\mc{A}$, cf.~\cite[Section~3.7]{ES}. So~this quaternionic structure exists if and only if $P\in \RN[x]$, $P(-x)=(-1)^n P(x)$.

\begin{Example} Let $n=2$, so $\mc A$ is the quotient
of the enveloping algebra $U(\g)$, $\g=\mathfrak{sl}_2$, by the relation $fe+\frac{(h+1)^2}{4}=C$, where
$C\in \mathbb R$. Since
\[
e=v,\qquad f=-u,\qquad h=2z,
\]
we have
\[
\rho_\pm(e)=\pm f,\qquad \rho_\pm(f)=\pm e,\qquad \rho_\pm(h)=-h.
\]
So $\g_+:=\g^{\rho_+}$ has basis
$\mathbf{x}=\frac{e+f}{2}$, $\mathbf{y}=\frac{{\rm i}(e-f)}{2}$, $\mathbf{z}=\frac{{\rm i}h}{2}$.
Thus,
\[
[\mathbf{x},\mathbf{y}]=-\mathbf{z},\qquad [\mathbf{z},\mathbf{x}]=\mathbf{y},\qquad [\mathbf{y},\mathbf{z}]=\mathbf{x}.
\]
Hence, setting $E:=\mathbf{y}-\mathbf{z}$, $F:=\mathbf{y}+\mathbf{z}$, $H:=2\mathbf{x}$,
we have
\[
[H,E]=2E,\qquad [H,F]=-2F,\qquad [E,F]=H,
\]
so $\g_+=\mathfrak{sl}_2(\mathbb R)$.

On the other hand, $\g_-:=\g^{\rho_-}$ has basis
${\rm i}\mathbf{x},{\rm i}\mathbf{y},\mathbf{z}$, hence $\g_-=\mathfrak{so}_3(\mathbb R)=\mathfrak{su}_2$.

So $\rho_+$ and $\rho_-$ correspond to the split and compact form of $\g$, respectively.
\end{Example}

\subsection{Twisted traces}
Let $\mc A=\mc A_P$ be a filtered quantization of $A$. Recall \cite[Section~3.1]{ES} that a $g_t$-{\it twisted trace} on~$\mc A$ is a linear map $T\colon \mc{A}\to \CN$ such that $T(ab)=T(bg_t(a))$, where $g_t$ is given by \eqref{gt}. It is shown in~\cite[Section~3]{ES}, that ($s$-invariant) nondegenerate twisted traces, up to scaling, correspond to ($s$-invariant) nondegenerate short star-products on~$\mc A$.

Let us classify $g_t$-twisted traces\footnote{One can show that for generic $P$ and $n>2$, the only possible filtration preserving automorphisms are $g_t$.} $T$ on~$\mc A$. The answer is given by the following proposition.

\begin{Proposition}
$T\colon\mc{A}\to \CN$ is a $g_t$-twisted trace on $\mc{A}$ if and only if
\begin{enumerate}\itemsep=0pt
\item[$(1)$]
$T(\mc{A}_j)=0$ for $j\ne 0$;
\item[$(2)$]
$T\big(S\big(z-\tfrac{1}{2}\big)P\big(z-\tfrac{1}{2}\big)\big)
=tT\big( S\big(z+\tfrac{1}{2}\big)P\big(z+\tfrac{1}{2}\big)\big)$ for all $S \in\CN[x]$.
\end{enumerate}
In particular, any twisted trace is automatically $s$-invariant.
\end{Proposition}

\begin{proof} Suppose $T$ satisfies (1), (2).
It is enough to check that
\[
T(ub)=t^{-1}T(bu),\qquad T(vb)=tT(bv),\qquad T(zb)=T(bz)
\]
for $b\in \mc{A}$.

The equality $T(zb)=T(bz)$ says that $T(\mc A_j)=0$ for $j\ne 0$, which is condition (1).

By (1), it is enough to check the equality $T(ub)=t^{-1} T(bu)$ for $b\in \mc{A}_{-1}$. In this case $b=vS\big(z+\tfrac{1}{2}\big)$ for some polynomial $S$. We~have
\begin{gather*}
T(ub)=T\big(uvS\big(z+\tfrac{1}{2}\big)\big)=T\big(P\big(z+\tfrac{1}{2}\big)S\big(z+\tfrac{1}{2}\big)\big),
\\
T(bu)=T\big(vS\big(z+\tfrac{1}{2}\big)u\big)=T\big(vu S\big(z-\tfrac{1}{2}\big)\big)=T\big(P\big(z-\tfrac{1}{2}\big)S\big(z-\tfrac{1}{2}\big)\big),
\end{gather*}
which yields the desired identity using (2).

Similarly, it is enough to check the equality $T(vb)=tT(bv)$ for $b\in \mc{A}_1$. In this case $b=u S\big(z-\tfrac{1}{2}\big)$. We~have
\begin{gather*}
T(vb)=T\big(vu S\big(z-\tfrac{1}{2}\big)\big)=T\big(P\big(z-\tfrac{1}{2}\big)S\big(z-\tfrac{1}{2}\big)\big),
\\
T(bv)=T\big(u S\big(z-\tfrac{1}{2}\big)v\big)=T\big(uv S\big(z+\tfrac{1}{2}\big)\big)=T\big(P\big(z+\tfrac{1}{2}\big)S\big(z+\tfrac{1}{2}\big)\big),
\end{gather*}
which again gives the desired identity using (2).

Conversely, the same argument shows that if $T$ is a $g_t$-twisted trace
then (1), (2) hold.
\end{proof}

Thus we get

\begin{Corollary}\label{natiso}
The space of $g_t$-twisted traces on~$\mc A$ is naturally isomorphic to the space
\[
\bigl(\CN[z]/\big\{S\big(z-\tfrac{1}{2}\big)P\big(z-\tfrac{1}{2}\big)-t S\big(z+\tfrac{1}{2}\big)P\big(z+\tfrac{1}{2}\big)\,|\, S \in\CN[z]\big\}\bigr)^*
\]
and has dimension $n$ if $t\neq 1$ and dimension $n-1$ if $t=1$.
\end{Corollary}

\subsection{The formal Stieltjes transform}
There is a useful characterization of the space $g_t$-twisted traces in
terms of generating functions. Given a linear functional $T$ on $\CN[z]$,
its {\em formal Stieltjes transform} is the generating function
\[
F_T(x):=\sum_{n\ge 0} x^{-n-1} T(z^n) \in \CN\big[\big[x^{-1}\big]\big],
\]
or equivalently $F_T(x)=T\big((x-z)^{-1}\big)$, with $(x-z)^{-1}$ itself expanded as a
formal power series in~$x^{-1}$.

\begin{Proposition}
 The formal Stieltjes transform of a $g_t$-twisted trace on~$\mc A$
 satisfies
 \[
 P(x)\big(F_T\big(x+\tfrac{1}{2}\big)-t F_T\big(x-\tfrac{1}{2}\big)\big)\in \CN[x],
 \]
 and this establishes an isomorphism of the space of $g_t$-twisted traces
 with the space of polynomials of degree $\le n-1$ $($for $t\ne 1)$ or $\le
 n-2$ $($for $t=1)$.
\end{Proposition}

\begin{proof}
 We may write
 \begin{gather*}
 P(x)\big(F_T\big(x+\tfrac{1}{2}\big) - t F_T\big(x-\tfrac{1}{2}\big)\big)
 = T\bigg(\frac{P(x)}{x+\frac{1}{2}-z}-t \frac{P(x)}{x-\frac{1}{2}-z}\bigg)
 \\ \hphantom{ P(x)\big(F_T\big(x+\tfrac{1}{2}\big) - t F_T\big(x-\tfrac{1}{2}\big)\big)}
 {}= T\bigg(\frac{P\big(z-\frac{1}{2}\big)}{x+\frac{1}{2}-z}-t \frac{P\big(z+\frac{1}{2}\big)}{x-\frac{1}{2}-z}\bigg)
 \\ \hphantom{ P(x)\big(F_T\big(x+\tfrac{1}{2}\big) - t F_T\big(x-\tfrac{1}{2}\big)\big)=}
{}+ T\bigg(\frac{P(x)-P\big(z-\frac{1}{2}\big)}{x-\big(z-\frac{1}{2}\big)}-t \frac{P(x)-P\big(z+\frac{1}{2}\big)}{x-\big(z+\frac{1}{2}\big)}\bigg).
 \end{gather*}
 In the final expression, the second term is the image under $T$ of a
 polynomial in $z$ and $x$, while the first term expands as
 \[
 \sum_{n\ge 0} x^{-n-1} T\big(P\big(z-\tfrac{1}{2}\big)\big(z-\tfrac{1}{2}\big)^n-t P\big(z+\tfrac{1}{2}\big)\big(z+\tfrac{1}{2}\big)^n\big) = 0.
 \]
 Since the map $F\mapsto F\big(x+\tfrac{1}{2}\big)-t F\big(x-\tfrac{1}{2}\big)$ is injective on $x^{-1}\mathbb C[[1/x]]$, this establishes an
 injective map from $g_t$-twisted traces to polynomials of degree
 $<\deg(P)$. This establishes the conclusion for $t\ne 1$, with surjectivity following by dimension count.
 Finally, for $t=1$, we need simply observe that for any $F\in x^{-1}\CN\big[\big[x^{-1}\big]\big]$,
 $F_T\big(x+\frac{1}{2}\big)-F_T\big(x-\tfrac{1}{2}\big)\in x^{-2}\CN\big[\big[x^{-1}\big]\big]$, and thus the polynomial has
 degree $<\deg(P)-1$, and surjectivity again follows from dimension count.
\end{proof}

\begin{Remark} It is easy to see that the map $F(x)\mapsto F\big(x+\frac{1}{2}\big)-t F\big(x-\tfrac{1}{2}\big)$ acts triangularly
 on $x^{-1}\CN\big[\big[x^{-1}\big]\big]$, of degree $0$ (with nonzero leading coefficients)
 if $t\ne 1$ and degree $-1$ (ditto) if $t=1$, letting one see directly
 that there is a unique solution of $P(x)\big(F\big(x+\frac{1}{2}\big)-t F\big(x-\tfrac{1}{2}\big)\big)=R(x)$
 for any polynomial $R$ satisfying the degree constraint.
\end{Remark}

\begin{Remark}
 A similar argument establishes an isomorphism between linear functionals
 satisfying $T\big(P\big(q^{-\frac{1}{2}}z\big)S\big(q^{-\frac{1}{2}}z\big)-qt P\big(q^{\frac{1}{2}}z\big)S\big(q^{\frac{1}{2}}z\big)\big)=0$ and
 elements $F\in x^{-1}\CN\big[\big[x^{-1}\big]\big]$ such that
 \[
 P(x)\big(F_T\big(q^{\frac{1}{2}}x\big)-t F_T\big(q^{-\frac{1}{2}}x\big)\big) \in \CN[x],
 \]
 or, for $t=1$, between linear functionals satisfying
 \[
 T\big(z^{-1}\big(P\big(q^{-\frac{1}{2}}z\big)S\big(q^{-\frac{1}{2}}z\big)
 -P\big(q^{\frac{1}{2}}z\big)S\big(q^{\frac{1}{2}}z\big)\big)\big)=0
 \]
 and formal series satisfying
 \[
 P(x) x^{-1}\big(F_T\big(q^{\frac{1}{2}}x\big)-F_T\big(q^{-\frac{1}{2}}x\big)\big)\in \CN[x].
 \]
\end{Remark}

\section{An analytic construction of twisted traces}\label{sec3}

\subsection[Construction of twisted traces when all roots of P(x) satisfy |Re alpha|<1/2]
{Construction of twisted traces when all roots of $\boldsymbol{P(x)}$ satisfy $\boldsymbol{|\Re\alpha|<\frac{1}{2}}$}
Let $t=\exp(2\pi {\rm i} c)$, where $0\le \Re c<1$ (clearly, such $c$ exists and is unique).

Let $P(x)=\prod_{j=1}^n(x-\alpha_j)$. Define
\[
\bP(X):=\prod_{j=1}^n\big(X+{\rm e}^{2\pi {\rm i}\alpha_j}\big).
\]
 When $P(x)$ satisfies the equation $\ovl{P}(-x)=(-1)^nP(x)$
 (the condition for existence of a conjugation $\rho$) the polynomial $\bP(X)$ has real coefficients.

\begin{Proposition}\label{classtr}
\label{PropTracesForOpenStrip}
Assume that every root $\alpha$ of $P(x)$ satisfies $|\Re\alpha|<\frac{1}{2}$. Also suppose first that $t$ does not belong to $\RN_{>0}\setminus\{1\}$, i.e., $\Re c\in (0,1)$ or $c=0$. Then every $g_t$-twisted trace is given by\footnote{Here $|{\rm d}x|$ denotes the Lebesgue measure on the imaginary axis.}
\[
T(R(z))=\int_{{\rm i}\RN} R(x)w(x)|{\rm d}x|, \qquad R\in \mathbb C[x],
\]
 where $w$ is the {\it weight function} defined by the formula
 \[
 w(x)=w(c,x):={\rm e}^{2\pi {\rm i}cx}\frac{G({\rm e}^{2\pi {\rm i} x})}{\bP({\rm e}^{2\pi {\rm i} x})},
 \]
 where $G$ is a polynomial of degree $\le n-1$ and $G(0)=0$ if $c=0$.
\end{Proposition}
\begin{proof}
It is easy to see that the function $w(x)$ enjoys the following properties:
\begin{enumerate}\itemsep=0pt
\item[$(1)$]
$w(x+1)=tw(x)$;
\item[$(2)$]
$|w(x)|$ decays exponentially and uniformly when $\Im x$ tends to $\pm\infty$;
\item[$(3)$]
$w\big(x+\frac{1}{2}\big)P(x)$ is holomorphic when $|\Re x|\leq \frac{1}{2}$.
\end{enumerate}

Indeed, (2) holds because the degree of $G$ is strictly less than the degree of $\mathbf{P}$ and either ${\rm Re}(c)>0$ or $G(0)=0$, and (3) holds because all roots of $P$ are in the strip $|{\rm Re}\alpha|<\frac{1}{2}$.

Let $T(R(z)):=\int_{{\rm i}\RN} R(x)w(x)|{\rm d}x|$. We~should check that
\[
T\big(tS\big(z+\tfrac{1}{2}\big)P\big(z+\tfrac{1}{2}\big)-S\big(z-\tfrac{1}{2}\big)P\big(z-\tfrac{1}{2}\big)\big)=0.
\]
We have \begin{gather*}
T\big(tS\big(z+\tfrac{1}{2}\big)P\big(z+\tfrac{1}{2}\big)-S\big(z-\tfrac{1}{2}\big)
P\big(z-\tfrac{1}{2}\big)\big)
\\ \qquad
{}=\int_{{\rm i}\RN}tS\big(x+\tfrac{1}{2}\big)P\big(x+\tfrac{1}{2}\big)w(x)|{\rm d}x|-\int_{{\rm i}\RN}S\big(x-\tfrac{1}{2}\big)P\big(x-\tfrac{1}{2}\big)w(x)|{\rm d}x|
\\ \qquad
{}=\int_{\frac{1}{2}+{\rm i}\RN}tS(x)P(x)w\big(x-\tfrac{1}{2}\big)|{\rm d}x|-\int_{-\frac{1}{2}+{\rm i}\RN}S(x)P(x)w\big(x+\tfrac{1}{2}\big)|{\rm d}x|
\\ \qquad
{}=\int_{\frac{1}{2}+{\rm i}\RN}S(x)P(x)w\big(x+\tfrac{1}{2}\big)|{\rm d}x|-\int_{-\frac{1}{2}+{\rm i}\RN}S(x)P(x)w\big(x+\tfrac{1}{2}\big)|{\rm d}x|
\\ \qquad
{}=\frac{1}{{\rm i}}\int_{\partial \left(\left[-\frac{1}{2},\frac{1}{2}\right]\times \RN\right)}S(x)P(x)w\big(x+\tfrac{1}{2}\big){\rm d}x.
\end{gather*}
But this integral vanishes by the Cauchy theorem since $S(x)P(x)w\big(x+\tfrac{1}{2}\big)$ is holomorphic when $|\Re x|\leq \frac{1}{2}$ and decays exponentially as $\Im x\to\pm {\rm i}\infty$.

By Corollary~\ref{natiso}, the space of polynomials $G(X)$ has the same dimension as the space of $g_t$-twisted traces, and the map sending polynomials $G$ to traces is clearly injective, so we have described all traces.
\end{proof}

Now consider the remaining case $t\in \mathbb R\setminus \lbrace 1\rbrace$, i.e., $c\in {\rm i}\mathbb R\setminus \lbrace 0\rbrace$. In this case the function $w(x)$ does not decay at $+{\rm i}\infty$, so the integral in Proposition~\ref{classtr} is not convergent. However, we can write the formula for $T(R(z))$ as follows, so that it makes sense in this case:
\[
T(R(z))=\lim_{\delta\to 0+}\int_{{\rm i}\RN} R(x)w(c+\delta,x)|{\rm d}x|.
\]
Alternatively, one may say that
$T(R(z))$ is the value of the Fourier transform of the distribution
$R(-{\rm i}y)w(0,-{\rm i}y)$ at the point ${\rm i}c$ (it is easy to see that this Fourier transform is given by an~analytic function outside of the origin). We~then have the following easy generalization of~Proposition~\ref{classtr}:

\begin{Proposition} With this modification, Proposition~$\ref{classtr}$ is valid for all $t$.
\end{Proposition}

Consider now the special case of even quantizations.
Recall \cite[Section~3.3]{ES} that nondegenerate {\it even} short star-products on $A$ correspond to nondegenerate $s$-twisted $\sigma$-invariant traces $T$ on~various even quantizations $\mc A$ of $A$, up to scaling. So~let us classify such traces. As shown above, $s$-twisted traces $T$ correspond to $w(x)$ such that $w(x+1)=(-1)^n w(x)$. Also, it is easy to see that such $T$ is $\sigma$-invariant if and only if $T(R(z))=T(R(-z))$. We~have
\[
T(R(-z))=\int_{{\rm i}\RN} R(-x)w(x) |{\rm d}x|=\int_{{\rm i}\RN} R(x)w(-x) |{\rm d}x|.
\]
So $T$ is $\sigma$-invariant of and only if $w(x)=w(-x)$. Thus we have the following proposition.

\begin{Proposition}
Suppose that $\mc{A}$ is an even quantization of $A$. Then $s$-twisted $\sigma$-invariant traces $T$ are given by the formula
\[
T(R(z))=\int_{{\rm i}\RN} R(x)w(x)|{\rm d}x|,
\]
where $w$ is as in Proposition~$\ref{PropTracesForOpenStrip}$ and
$w(x)=w(-x)=(-1)^n w(x+1)$.
\end{Proposition}

\subsection{Relation to orthogonal polynomials}\label{ortpol}

We continue to assume that all roots of $\mathbf{P}$ are in the strip $|{\rm Re}\alpha|<\frac{1}{2}$. Assume that the trace $T$ is nondegenerate, i.e., the form $(a,b)\mapsto T(ab)$ defines an inner product on~$\mc A$ nondegenerate on~each filtration piece.
This holds for generic parameters, e.g., specifically if $w(x)$ is nonnegative on~${\rm i}\mathbb R$. Let~$\phi\colon A\to \mc A$ be the quantization map defined by $T$
(see~\cite[Section~3]{ES}). Namely, the form $(a,b)$ allows us to split the filtration, and $\phi$ is precisely the splitting map. Thus, $\phi(z^k)=p_k(z)$, where
$p_k$ are monic orthogonal polynomials
for the inner product
\[
(f_1,f_2)_*:=\int_{{\rm i}\mathbb R}f_1(x)f_2(x)w(x)|{\rm d}x|.
\]
Recall~\cite{Sz} that these polynomials satisfy a 3-term recurrence
\[
p_{k+1}(x)=(x-b_k)p_k(x)-a_kp_{k-1}(x),
\]
for some numbers $a_k$, $b_k$, i.e.,
\[
xp_{k}(x)=p_{k+1}(x)+b_kp_k(x)+a_kp_{k-1}(x).
\]
Thus the corresponding short star-product $z*z^k$
has the form
\begin{gather*}
z*z^k=\phi^{-1}\big(\phi(z)\phi\big(z^k\big)\big)=\phi^{-1}(zp_k(z))=
\phi^{-1}(p_{k+1}(z)+b_kp_k(z)+a_kp_{k-1}(z))
\\ \hphantom{z*z^k}
{}=z^{k+1}+b_kz^k+a_kz^{k-1}.
\end{gather*}
Thus the numbers $a_k$, $b_k$ are the matrix elements of multiplication by $z$ in weight $0$ for the short star-product attached to $T$. More general matrix elements of multiplication by $u$, $v$, $z$ for this short star-product are computed similarly. In other words, to compute the short star-product attached to $T$, we need to compute explicitly the coefficients $a_k$, $b_k$ and their generalizations. This problem
is addressed in Section~\ref{expcom}.

It is more customary to consider orthogonal polynomials on the real (rather than imaginary) axis, so let us make a change of variable $x=-{\rm i}y$.
Then we see that the monic polynomials $P_k(y):={\rm i}^{k}p_k(-{\rm i}y)$
are orthogonal under the inner product
\[
(f_1,f_2):=\int_{-\infty}^\infty f_1(y)f_2(y){\rm w}(y){\rm d}y,
\]
where ${\rm w}(y):=w(-{\rm i}y)$. Then the 3-term recurrence looks like
\[
P_{k+1}(y)=(y-{\rm i}b_k)P_k(y)+a_kP_{k-1}(y)
\]
(so for real parameters we'll have $a_k\in \mathbb R$, $b_k\in {\rm i}\mathbb R$).

\begin{Example}\label{n=1}
Let $n=1$, $P(x)=x$, so $\mathbf{P}(X)=X+1$.
Then a nonzero twisted trace exists if and only if $c\ne 0$, in which case it is unique up to scaling, and the corresponding weight func\-tion~is
\[
w(x)=\frac{{\rm e}^{2\pi {\rm i}cx}}{{\rm e}^{2\pi {\rm i}x}+1}
=\frac{{\rm e}^{2\pi {\rm i}\left(c-\frac{1}{2}\right)x}}{2\cos \pi x},\qquad
{\rm w}(y)=\frac{{\rm e}^{2\pi\left(c-\frac{1}{2}\right)y}}{2\cosh \pi y}.
\]
The corresponding orthogonal polynomials $P_k(y)$ are the (monic) {\it Meixner--Pollaczek polynomials} with parameters $\lambda=\frac{1}{2}$, $\phi=\pi c$~\cite[Section~1.7]{KS}.
\end{Example}

\begin{Example}\label{n=2} Let $n=2$, $P(x)=x^2+\beta^2$, so
\[
\mathbf{P}(X)=\big(X+{\rm e}^{2\pi \beta}\big)\big(X+{\rm e}^{-2\pi \beta}\big).
\]
The space of twisted traces is $1$-dimensional if $c=0$ and $2$-dimensional if $c\ne 0$. So~for $c\ne 0$ the traces up to scaling
are defined by the weight function
\[
w(x)=\frac{{\rm e}^{2\pi {\rm i}\left(c-\frac{1}{2}\right)x} \cos\pi (x-{\rm i}\alpha)}{2\cos \pi(x-{\rm i}\beta )\cos \pi(x+{\rm i}\beta )},\qquad {\rm w}(y)=\frac{{\rm e}^{2\pi\left(c-\frac{1}{2}\right)y} \cosh\pi (y-\alpha)}{2\cosh \pi(y-\beta)\cosh \pi(y+\beta)},
\]
and the limiting cases $\alpha\to \pm \infty$ along the real axis, which yield
\[
w(x)=\frac{{\rm e}^{2\pi {\rm i}\left(c-\frac{1}{2}\pm \frac{1}{2}\right)x}}{4\cos \pi(x-{\rm i}\beta )\cos \pi(x+{\rm i}\beta )},\qquad {\rm w}(y)=\frac{{\rm e}^{2\pi\left(c-\frac{1}{2}\pm \frac{1}{2}\right)y}}{4\cosh \pi(y-\beta)\cosh \pi(y+\beta)}.
\]

These formulas for the plus sign also define the unique up to scaling trace for $c=0$; i.e., \[
w(x)=\frac{1}{4\cos\pi(x-{\rm i}\beta )\cos\pi(x+{\rm i}\beta )},\qquad {\rm w}(y)=\frac{1}{4\cosh\pi(y-\beta)\cosh\pi(y+\beta)}.
\] In this case, the corresponding orthogonal polynomials $P_k(y)$ are the {\it continuous Hahn polynomials} with parameters $\tfrac{1}{2}+{\rm i}\beta ,\tfrac{1}{2}-{\rm i}\beta ,\tfrac{1}{2}-{\rm i}\beta ,\tfrac{1}{2}+{\rm i}\beta$~\cite[Section~1.4]{KS}.

Also for $c=\frac{1}{2}$, $\alpha=0$ we have
\[
w(x)=\frac{\cos \pi x}{2\cos\pi(x+{\rm i}\beta )\cosh\pi(x-{\rm i}\beta )},\qquad {\rm w}(y)=\frac{\cosh \pi y}{2\cosh\pi(y+\beta)\cosh\pi(y-\beta)},
\]
so $P_k(y)$ are the {\it continuous dual Hahn polynomials} with $a=0$, $b=\frac{1}{2}-{\rm i}\beta $, $c=\frac{1}{2}+{\rm i}\beta$ \cite[Section~1.3]{KS}.
\end{Example}

\begin{Remark} In Example~\ref{n=1} ($n=1$),
the only even short star-product corresponds
to ${\rm w}(y)=\frac{1}{2\cosh \pi y}$. This is the Moyal--Weyl star-product.
In Example~\ref{n=2} ($n=2$), the only even short star-product corresponds
to ${\rm w}(y)=\frac{1}{4\cosh \pi (y-\beta)\cosh \pi(y+\beta)}$. This is the
unique ${\rm SL}_2$-invariant star-product.
\end{Remark}

\begin{Example}\label{betai} Let $t=(-1)^n$, $G(X)=X^{[n/2]}$. Then
\[
w(x)=\prod_{j=1}^n\frac{1}{2\cos \pi(x-{\rm i}\beta _j)},\qquad
{\rm w}(y)=\prod_{j=1}^n\frac{1}{2\cosh \pi(y+\beta_j)},
\]
which defines an $s$-twisted trace.
The corresponding orthogonal polynomials are semiclassical but not hypergeometric for $n\ge 3$.
\end{Example}

\begin{Remark} The trace of Example~\ref{betai} corresponds to the short star-product arising in the 3-d SCFT, as shown in~\cite[Section~8.1.2]{DPY}. There the Kleinian singularity of type $A_{n-1}$ appears as the Higgs branch, and the parameters $\beta_j$ are the FI parameters. The same trace also shows up in~\cite[equation~(5.27)]{DFPY}, where
the Kleinian singularity appears as the Coulomb branch, and the parameters $\beta_j$ are the mass parameters.\footnote{We thank Mykola Dedushenko for this explanation.}
\end{Remark}

\subsection{Conjugation-equivariant traces}
Let now $\rho$ be a conjugation on~$\mc A$ (Section~\ref{conju}). Let~us determine which $g_t$-twisted traces are $\rho$-equivariant (see~\cite[Section~3.6]{ES}. A trace $T$ is $\rho$-equivariant if $\ovl{T(R(z))}=T\big(\ovl{R}(-z)\big)$, which is~equivalent to $T$ being real on $\RN[{\rm i}z]$. This happens if and only if $w(x)$ is real on ${\rm i}\RN$. Since $w$ is meromorphic this means that $w(x)=\ovl{w(-\ovl{x})}$.

So we have the following proposition.

\begin{Proposition}
\label{PropQuaternionicTraces}
Suppose that $\mc{A}$ is a quantization of $A$ with conjugation $\rho$. Then $\rho$-equivariant $g_t$-twisted traces $T$ on~$\mc A$ are
given by
\[
T(R(z))=\int_{{\rm i}\RN} R(x)w(x)|{\rm d}x|,
\]
where $w$ is as in Proposition~$\ref{PropTracesForOpenStrip}$ and
\[w(x)=\ovl{w(-\ovl{x})}=(-1)^n w(x+1).
\]
Moreover, if $\mc A$ is even then $\sigma$-invariant traces among them
correspond to the functions $w$ with $w(x)=w(-x)$.
\end{Proposition}

\subsection[Construction of traces when all roots of P(x) satisfy |Re alpha| < 1/2]
{Construction of traces when all roots of $\boldsymbol{P(x)}$ satisfy $\boldsymbol{|\Re\alpha|\leq \frac{1}{2}}$}

From now on we suppose that ${\rm i}^nP(x)$ is real on ${\rm i}\RN$ (so that the conjugations $\rho_\pm$ are well defined). In particular, the roots of $P(x)$ are symmetric with respect to ${\rm i}\RN$.

Suppose that for all roots $\alpha$ of $P(x)$ we have $|\Re\alpha|\leq \frac{1}{2}$, and let us give a formula for twisted traces in this case.
 There are unique monic polynomials $P_*(x)$, $Q(x)$ such that $P(x)=P_*(x)Q\big(x+\frac{1}{2}\big)Q\big(x-\tfrac{1}{2}\big)$, all roots of $P_*(x)$ belong to the strip $|\Re x|<\frac{1}{2}$ and all roots of $Q(x)$ belong to ${\rm i}\RN$. Suppose that $\alpha_1,\ldots,\alpha_k$ are the roots of $P_*(x)$ and $\alpha_{k+1},\ldots,\alpha_m$ are the roots of~$Q\big(x+\frac{1}{2}\big)$. Note that $\deg Q=n-m$. Let~$\bP_*(X)=\prod_{j=1}^m(X+{\rm e}^{2\pi {\rm i}\alpha_j})$, $w(x)={\rm e}^{2\pi {\rm i}cx}\frac{G({\rm e}^{2\pi {\rm i} x})}{\bP_*({\rm e}^{2\pi {\rm i} x})}$, where $G(X)$ is a polynomial of degree at most $m-1$ and $G(0)=0$ when $t=1$. We~have
\begin{enumerate}\itemsep=0pt
\item[(1)]
$w(x+1)=tw(x)$;
\item[(2)]
$w(x)Q(x)$ is bounded on ${\rm i}\RN$ and decays exponentially and uniformly when $\Im x$ tends to~$\pm \infty$;
\item[(3)]
$w\big(x+\frac{1}{2}\big)P(x)$ is holomorphic on $|\Re x|\leq \frac{1}{2}$.
\end{enumerate}
For any $R\in \CN[x]$ let $R(x)=R_1(x)Q(x)+R_0(x)$, where $\deg R_0<\deg Q$.
\begin{Proposition}
\label{PropTracesForClosedStrip}
A general $g_t$-twisted trace on~$\mc A$
has the form
\[T(R(z))=\int_{{\rm i}\RN}R_1(x)Q(x)w(x)|{\rm d}x|+\phi(R_0),\] where $w(x)$ is as above and $\phi$ is any linear functional.
\end{Proposition}
\begin{proof}
The space of polynomials $G$ has dimension $m-\delta_c^0$, while the space of linear functionals~$\phi$ has dimension $\deg Q=n-m$. So~the space of such linear functionals $T$ has dimension $n-\delta_c^0$. The space of all $g_t$-twisted traces has the same dimension, so it is enough to prove that all linear functionals $T$ of this form are $g_t$-twisted traces. In other words, we should prove that $T\big(S\big(z-\frac{1}{2}\big)P\big(z-\frac{1}{2}\big)-tS\big(z+\frac{1}{2}\big)P\big(z+\frac{1}{2}\big)\big)=0$ for all $S\in\CN[x]$.

We see that $S\big(x-\tfrac{1}{2}\big)P\big(x-\tfrac{1}{2}\big)-tS\big(x+\frac{1}{2}\big)P\big(x+\frac{1}{2}\big)$ is divisible by $Q(x)$, so
\begin{gather*}
T\big(S\big(z-\tfrac{1}{2}\big)P\big(z-\tfrac{1}{2}\big)-tS\big(z+\tfrac{1}{2}\big)
P\big(z+\tfrac{1}{2}\big)\big)
\\ \qquad
{}= \int_{{\rm i}\RN}\big(S\big(x-\tfrac{1}{2}\big)P\big(x-\tfrac{1}{2}\big)-tS\big(x+\tfrac{1}{2}\big)
P\big(x+\tfrac{1}{2}\big)\big)w(x)|{\rm d}x|.
\end{gather*}
Since $w\big(x+\frac{1}{2}\big)P(x)$ is holomorphic on $|\Re x|\leq \frac{1}{2}$, we deduce that this integral is zero similarly to the proof of Proposition~\ref{PropTracesForOpenStrip}
\end{proof}

\subsection{Twisted traces in the general case}
\label{SubSubSecGeneralTraces}
Let $m(\alpha)$ be the multiplicity of $\alpha$ as a root of $P(x)$. Any linear functional $\phi$ on the space~$\CN[x]/$ $P(x)\CN[x]$ can be written as $\phi(S)=\suml\nolimits_{\alpha,\,0\le i<m(\alpha)}C_{\alpha i} S^{(i)}(\alpha)$, where $C_{\alpha i}\in\CN$. Therefore any $g_t$-twis\-ted trace $T$ is given by
\[
T\big(S\big(z-\tfrac{1}{2}\big)-tS\big(z+\tfrac{1}{2}\big)\big)
=\suml_{\alpha,\,0\le i<m(\alpha)}C_{\alpha i} S^{(i)}(\alpha).
\]

Let $\widetilde{P}(x)$ be the following polynomial: all roots of $\widetilde{P}(x)$ belong to the strip $|\Re x|\leq \frac{1}{2}$ and the multiplicity of a root $\alpha$ equals to
\begin{itemize}
\item
$\sum_{k\in \ZN} m(\alpha+k)$ if $|\Re\alpha|<\frac{1}{2}$;
\item
$\sum_{k\geq 0} m(\alpha+k)$ if $\Re\alpha=\frac{1}{2}$;
\item
$\sum_{k\leq 0} m(\alpha+k)$ if $\Re\alpha=-\frac{1}{2}$.
\end{itemize}
So $\widetilde{P}(x)$ has the same degree as $P(x)$ and its roots are obtained from roots of $P(x)$ by the minimal integer shift into the strip $|\Re x|\leq \frac{1}{2}$. In particular, the roots of $\widetilde{P}(x)$ are symmetric with respect to ${\rm i}\RN$.

Suppose that $\alpha\in \CN$ has real part bigger than $\frac{1}{2}$, $S(x)$ is an arbitrary polynomial, $R(x)=S\big(x-\tfrac{1}{2}\big)-tS\big(x+\frac{1}{2}\big)$, $i\geq 0$. Let~$r$ be the smallest positive integer such that ${\rm Re}(\alpha)-r\leq \frac 12$. Then
\begin{gather*}
S^{(i)}(\alpha)=\sum_{k=0}^{r-1}\big(t^{-k}S^{(i)}(\alpha-k)
-t^{-k-1}S^{(i)}(\alpha-k-1)\big)+t^{-r}S^{(i)}(\alpha-r)
\\ \hphantom{S^{(i)}(\alpha)}
{}=\sum_{k=0}^{r-1}t^{-k-1}R^{(i)}\big(\alpha-k-\tfrac{1}{2}\big)+t^{-r}S^{(i)}(\alpha-r)
=\phi_{i,\alpha}(R)+t^{-r}S^{(i)}(\alpha-r),
\end{gather*}
where
\[
\phi_{i,\alpha}(R):=\sum_{k=0}^{r-1}t^{-k-1}R^{(i)}\big(\alpha-k-\tfrac{1}{2}\big).
\]
We can write a similar equation for $\alpha\in \CN$ with real part smaller than $-\frac{1}{2}$.

Therefore
\begin{gather*}
T\big(S\big(z-\tfrac{1}{2}\big)-tS\big(z+\tfrac{1}{2}\big)\big)=\!\!\!\sum_{\alpha,\,0\le i<m(\alpha)}\!\!\!\!\!C_{\alpha i} S^{(i)}(\alpha)
\\ \hphantom{T\big(S\big(z-\tfrac{1}{2}\big)-tS\big(z+\tfrac{1}{2}\big)\big)}
{}=\!\!\!\sum_{\alpha,\,0\le i<m(\alpha)}\!\!\!\!\!C_{\alpha i} \phi_{i,\alpha}(R)+t^{-r(\alpha)}C_{\alpha i} S^{(i)}(\alpha-r)=\Phi(R)+\widetilde{T}(R(z)),
\end{gather*}
where $\Phi(R):=\suml\nolimits_{\Re a\neq 0,k\ge 0}c_{ak} R^{(k)}(a)$, $c_{ak}\in \CN$, $\widetilde{T}$ is a $g_t$-twisted trace for the quantization defined by the polynomial $\widetilde{P}(x)$. Below we will abbreviate this sentence to ``$\widetilde{T}$ is a trace for $\widetilde{P}$''.

Let $P_\circ$ be the following polynomial: all the roots of $P_\circ$ belong to strip $|\Re x|\leq \frac{1}{2}$ and the multiplicity of $\alpha$, $|\Re\alpha|\leq \frac{1}{2}$ in $P_\circ$ equals the multiplicity of $\alpha$ in $P$.

Since $\phi_{i,\alpha}$ are linearly independent for different $i,\alpha$, we deduce that $\Phi=0$ if and only if $T$ is a trace for $P_\circ$.

So we have proved the following proposition:
\begin{Proposition}
\label{PropGeneralTrace}
Suppose that $P$ is any polynomial, $\widetilde{P}$ is obtained from $P$ by the minimal integer shift of roots into the strip $|\Re x|\leq \frac{1}{2}$, and $P_\circ$ is obtained from $P$ by throwing out roots not in the strip $|\Re x|\leq \frac{1}{2}$. Then any twisted trace $T$ on $\mathcal A_P$ can be represented as $T=\Phi+\widetilde{T}$, where
\[
\Phi(R)=\sum_{a\notin {\rm i}\RN,\,k\ge 0} c_{a k} R^{(k)}(a),
\]
and $\widetilde{T}$ is a trace for $\widetilde{P}$. Furthermore, if $\Phi=0$ then $T$ is a trace for $P_\circ$.
\end{Proposition}

\begin{Remark} We may think about Proposition~\ref{PropGeneralTrace} as follows.
When the roots of $P$ lie inside the strip $|\Re x|<\frac{1}{2}$, the trace of $R(z)$ is given by
the integral of $R$ against the weight function~$w$ along the imaginary axis.
However, when we vary~$P$, as soon as its roots leave the strip $|\Re x|<\frac{1}{2}$,
poles of~$w$ start crossing the contour of integration. So~for the formula to remain valid,
we need to add the residues resulting from this. These residues give rise to the linear functional $\Phi$.
\end{Remark}

\section{Positivity of twisted traces}\label{sec4}

\subsection{Analytic lemmas}
We will use the following classical result:
\begin{Lemma}
\label{LemClassicalDense}
Suppose that $w(x)\ge 0$ is a measurable function on the real line such that $w(x)<c {\rm e}^{-b|x|}$ for some $c,b>0$. We~also assume that $w>0$ almost everywhere. Then polynomials are dense in the space $L^p(\RN,w(x){\rm d}x)$ for all $1\leq p<\infty$.
\end{Lemma}
\begin{proof}
Changing $x$ to $bx$ we can asssume that $b=1$.

Fix $p$. Let~$\frac{1}{p}+\frac{1}{q}=1$. Since $L^p(\RN,w)^*=L^q(\RN,w)$, it suffices to show that any function $f\in L^q(\RN,w)$ such that $\int_{\RN} f(x)x^n w(x){\rm d}x=0$ for all nonnegative integers $n$ must be zero.

Choose $0<a<\frac{1}{p}$. We~have ${\rm e}^{a|x|}\in L^p(\RN,w)$. Therefore $f(x){\rm e}^{a|x|}w(x)\in L^1(\RN)$. Denote $f(x)w(x)$ by $F(x)$. Let~$\widehat{F}$ be the Fourier transform of $F$. Since $F(x){\rm e}^{a|x|}\in L^1(\RN)$, $\widehat{F}$ extends to a holomorphic function in the strip $|\Im x|<a$.

Since $\int_{\RN}f(x)x^n w(x){\rm d}x=0$, we have $\int_{\RN}F(x)x^n{\rm d}x=0$, so $\widehat{F}^{(n)}(0)=0$. Since $\widehat{F}$ is a holomorphic function and all derivatives of $\widehat{F}$ at the origin are zero, we deduce that $\widehat{F}=0$. Therefore $F=0$, so $f=0$ almost everywhere, as desired.
\end{proof}
We get the following corollaries:
\begin{Lemma}\label{CorClassical}
Let $w$ satisfy the assumptions of Lemma~$\ref{LemClassicalDense}$.
\begin{enumerate}\itemsep=0pt
\item[$1.$]
Suppose that $H(x)$ is a continuous complex-valued function on $\mathbb R$ with finitely many zeros and at most polynomial growth at infinity. Then the set $\{H(x)S(x)\,|\, S(x)\in\CN[x]\}$ is dense in the space $L^p(\RN,w)$.
\item[$2.$]
Suppose that $M(x)$ is a nonzero polynomial nonnegative on the real line. Then the closure of the set $\{M(x)S(x)\ovl{S}(x)\,|\, S(x)\in \CN[x]\}$ in $L^p(\RN,w)$ is the subset of almost everywhere nonnegative functions.
\end{enumerate}
\end{Lemma}

\begin{proof}
1.\ The function $w(x)|H(x)|^p$ satisfies the assumptions of Lemma~\ref{LemClassicalDense}. Therefore polynomials are dense in the space $L^p(\RN,w|H|^p)$. The map $g\mapsto gH$ is an isometry between $L^p(\RN,w|H|^p)$ and $L^p(\RN,w)$. The statement follows.

2.\ Suppose that $f\in L^p(\RN,w)$ is nonnegative almost everywhere. Then $\sqrt{f}$ is an element of~$L^{2p}(\RN,w)$. Using (1), we find a sequence $S_n\in \CN[x]$ such that $\sqrt{M}S_n$ tends to $\sqrt{f}$ in $L^{2p}(\RN,w)$. We~use the following corollary of Cauchy--Schwarz inequality: if $a_k$, $b_k$ tend to $a$,~$b$ respectively in $L^{2p}(\RN,w)$ then $a_kb_k$ tends to $ab$ in $L^p(\RN,w)$.
Applying this to $a=b=\sqrt{f}$, $a_n=\sqrt{M}S_n$, $b_k=\sqrt{M}\ovl{S_n}$ we deduce that $MS_n\ovl{S_n}$ tends to $f$ in $L^p(\RN,w)$. The statement follows.
\end{proof}

\subsection[The case when all roots of P(x) satisfy |Re alpha|<1/2]
{The case when all roots of $\boldsymbol{P(x)}$ satisfy $\boldsymbol{|\Re\alpha|<\frac{1}{2}}$}

Let $\mc A$ be a filtered quantization of $A$ with conjugations $\rho_\pm$ such that $\rho_\pm^2=g_t$. We~want to classify positive definite Hermitian $\rho_\pm$-invariant forms on $\mc{A}$, i.e., positive definite Hermitian forms $(\cdot,\cdot)$ on $\mc{A}$ such that
\[
(a\rho(y),b)=(a,yb)
\]
 for all $a,b,y\in\mc{A}$, where $\rho=\rho_\pm$.

In this subsection we will do the classification in the case when all roots $\alpha$ of $P(x)$ satisfy $|\Re\alpha|<\tfrac{1}{2}$. We~start with general results that are true for all parameters $P$.

 It is easy to see that Hermitian $\rho$-invariant forms are in one-to-one correspondence with $g_t$-twisted $\rho$-invariant traces, i.e., $g_t$-twisted traces $T$ such that $T(\rho(a))=\overline{T(a)}$. The correspondence is as follows:
 \[(a,b)=T(a\rho(b)),\qquad T(a)=(a,1).\] Therefore it is enough to classify $g_t$-twisted traces $T$ such that the Hermitian form $(a,b)=T(a\rho(b))$ is positive definite. This means that $T(a\rho(a))>0$ for all nonzero $a\in \mc{A}$. Recall that $\ad z$ acts on $\mc{A}$ diagonalizably, $\mc{A}=\oplus_{d\in \ZN} \mc{A}_{d}$. Thus it is enough to check the condition $T(a\rho(a))>0$ for homogeneous $a$.

\begin{Lemma}\label{posde}\qquad
\begin{enumerate}\itemsep=0pt
\item[$1.$]
$T$ gives a positive definite form if and only if one has $T(a\rho(a))>0$ for all nonzero $a\in \mc{A}$ of weight $0$ or $1$.
\item[$2.$]
$T$ gives a positive definite form if and only if
\[T\big(R(z)\ovl{R}(-z)\big)>0\qquad \text{and}\qquad
\lambda T\big(R\big(z-\tfrac{1}{2}\big)\ovl{R}\big(\tfrac{1}{2}-z\big)P\big(z-\tfrac{1}{2}\big)\big)>0
\]
for all nonzero $R\in \CN[x]$.
\end{enumerate}
\end{Lemma}

\begin{proof}
1.\ Suppose that $T(a\rho(a))>0$ for all nonzero $a\in \mc{A}$ of weight $0$ or $1$. Let~$a$ be a nonzero homogeneous element of $\mc{A}$ with positive weight. There exists $b$ of weight $0$ or $1$ and nonnegative integer $k$ such that $a=v^kbv^k$. We~have
\begin{gather*}
T(a\rho(a))=\lambda^{2k}T\big(v^kbv^ku^k\rho(b)u^k\big)=\lambda^{2k}T\big(g_t^{-1}\big(u^k\big)v^kbv^ku^k\rho(b)\big)
\\ \hphantom{T(a\rho(a))}
{}= \lambda^{2k}t^{k}T\big(u^kv^kbv^ku^k\rho(b)\big)=
(-1)^{nk}T\big(u^kv^kbv^ku^k\rho(b)\big)=T\big(u^kv^kb\rho\big(u^kv^kb\big)\big)>0
\end{gather*}
since $u^kv^kb$ is a homogeneous element of weight $0$ or $1$.

Suppose that $a$ is a nonzero homogeneous element of $\mc{A}$ with negative weight. Then $a=\rho(b)$, where $b$ is a homogeneous element with positive weight. We~get
\[
T(a\rho(a))=T(\rho(b)\rho^2(b))=T(\rho(b)g_t(b))=T(b\rho(b))>0.
\]

2.\ Suppose that $a$ is an element of $\mc{A}_0$. Then $a=R(z)$ for some $R\in\CN[x]$. We~have $T(a\rho(a))=T(R(z)\ovl{R}(-z))$.

Suppose that $a$ is an element of $\mc{A}_1$. Then $a=R\big(z-\frac{1}{2}\big)v$ for some $R\in \CN[x]$. We~have
\begin{gather*}
T(a\rho(a))=\lambda T\big(R\big(z-\tfrac{1}{2}\big)v \ovl{R}\big(-z-\tfrac{1}{2}\big)u\big)
= \lambda T\big(R\big(z-\tfrac{1}{2}\big)vu \ovl{R}\big(-z+\tfrac{1}{2}\big)\big)
\\ \hphantom{T(a\rho(a))}
{}=
\lambda T\big(R\big(z-\tfrac{1}{2}\big)\ovl{R}\big(\tfrac{1}{2}-z\big)P\big(z-\tfrac{1}{2}\big)\big).
\end{gather*}

The statement follows.
\end{proof}

\begin{Proposition}
\label{PropFromPositiveTraceToFunctions}
Suppose that $T(R(z))=\int_{{\rm i}\RN}R(x)w(x)|{\rm d}x|$. Then $T$ gives positive definite form if and only if $w(x)$ and $\lambda w\big(x+\frac{1}{2}\big)P(x)$ are nonnegative on ${\rm i}\RN$.
\end{Proposition}
\begin{proof} By Lemma~\ref{posde}
$T$ gives positive definite form if and only if
\[
T(R(z)\ovl{R}(-z))>0
\]
and
\[
\lambda T\big(R\big(z-\tfrac{1}{2}\big)\ovl{R}\big(\tfrac{1}{2}-z\big)P\big(z-\tfrac{1}{2}\big)\big)>0
\]
for all nonzero $R\in \CN[x]$. A polynomial $S\in \CN[x]$ can be represented as $S(x)=R(x)\ovl{R}(-x)$ if~and only if $S$ is nonnegative on ${\rm i}\RN$. So~we have $T(R(z)\ovl{R}(-z))>0$ for all nonzero $R\in \CN[x]$ if~and only if
\[
\int_{{\rm i}\RN}S(x)w(x)|{\rm d}x|>0
\]
for all nonzero $S\in \CN[x]$ nonnegative on ${\rm i}\RN$. Using Lemma~\ref{CorClassical}(2) for $M=1$, we see that this is equivalent to $w(x)$ being nonnegative on ${\rm i}\RN$.

We have
\begin{gather*}
T\big(R\big(z-\tfrac{1}{2}\big)\ovl{R}\big(\tfrac{1}{2}-z\big)P\big(z-\tfrac{1}{2}\big)\big)=\int_{{\rm i}\RN}R\big(x-\tfrac{1}{2}\big)\ovl{R}\big(\tfrac{1}{2}-x\big)P\big(x-\tfrac{1}{2}\big)w(x)|{\rm d}x|
\\ \hphantom{T\big(R\big(z-\tfrac{1}{2}\big)\ovl{R}\big(\tfrac{1}{2}-z\big)P\big(z-\tfrac{1}{2}\big)\big)}
{}=\int_{\frac{1}{2}+{\rm i}\RN}R(x)\ovl{R}\big(-x\big)P(x)w\big(x+\tfrac{1}{2}\big)|{\rm d}x|
\\ \hphantom{T\big(R\big(z-\tfrac{1}{2}\big)\ovl{R}\big(\tfrac{1}{2}-z\big)P\big(z-\tfrac{1}{2}\big)\big)}
{}=\int_{{\rm i}\RN}R(x)\ovl{R}\big(-x\big)P(x)w\big(x+\tfrac{1}{2}\big)|{\rm d}x|.
\end{gather*}
In the last equality we used the Cauchy theorem and the fact that the function \mbox{$P(x)w\big(x+\tfrac{1}{2}\big)$} is~holomorphic when $|\Re x|\leq \frac{1}{2}$. Using Lemma~\ref{CorClassical}(2) for $M=1$ again, we see that \mbox{$\lambda T\big(R\big(z-\frac{1}{2}\big)$} $\times\ovl{R}\big(\frac{1}{2}-z\big) P\big(z-\frac{1}{2}\big)\big)>0$ for all nonzero $R\in \CN[x]$ if and only if $\lambda P(x)w\big(x+\frac{1}{2}\big)$ is nonnegative on ${\rm i}\RN$.
\end{proof}

From now on we assume that all roots $\alpha$ of $P(x)$ satisfy $|\Re\alpha|<\frac{1}{2}$. In this case every trace~$T$ can be represented as $T(R(z))=\int_{{\rm i}\RN}R(x)w(x)|{\rm d}x|$.
Recall that $w(x)={\rm e}^{2\pi {\rm i}cx}\frac{G({\rm e}^{2\pi {\rm i} x})}{\bP({\rm e}^{2\pi {\rm i} x})},$ where~$G$ is any polynomial with $\deg G\leq \deg P$ in the case when $c\neq 0$ and $\deg G<\deg P$ in the case when~$c=0$.

\begin{Proposition}
\label{PropFromFunctionsToPolynomials}\qquad
\begin{enumerate}\itemsep=0pt
\item[$1.$]
If $\lambda=-{\rm i}^{-n}{\rm e}^{-\pi {\rm i}c}$ $($i.e., $\rho=\rho_-)$ then $w(x)$ and $\lambda P(x)w\big(x+\frac{1}{2}\big)$ are nonnegative on ${\rm i}\RN$ if~and only if $G(X)$ is nonnegative when $X>0$ and nonpositive when $X<0$.
\item[$2.$]
If $\lambda=+{\rm i}^{-n}{\rm e}^{-\pi {\rm i}c}$ $($i.e., $\rho=\rho_+)$ then $w(x)$ and $\lambda w\big(x+\frac{1}{2}\big)P(x)$ are nonnegative on ${\rm i}\RN$ if~and only if $G(X)$ is nonnegative for all $X\in \RN$.
\end{enumerate}

\end{Proposition}
\begin{proof}
It is easy to see that $\bP(X)$ is positive when $X>0$. Therefore $w(x)$ is nonnegative on~${\rm i}\RN$ if and only if $G(X)$ is nonnegative when $X>0$.

We have
\[
\lambda P(x)w\big(x+\frac{1}{2}\big)=\pm {\rm i}^{-n}P(x) {\rm e}^{2\pi {\rm i} cx}\frac{G(-{\rm e}^{2\pi {\rm i} x})}{\bP(-{\rm e}^{2\pi {\rm i} x})}.
\]
It is clear that $\frac{{\rm i}^{-n}P(x)}{\bP(-{\rm e}^{2\pi {\rm i} x})}$ belongs to $\RN$ when $x\in {\rm i}\RN$ and does not change sign on ${\rm i}\RN$. When~$x$ tends to $-{\rm i}\infty$, the functions ${\rm i}^{-n}P(x)$ and $\bP(-{\rm e}^{2\pi {\rm i} x})$ have sign $(-1)^n$. Therefore $\frac{{\rm i}^{-n}P(x)}{\bP({\rm e}^{-2\pi {\rm i}x})}$ is positive on ${\rm i}\RN$. We~deduce that $\pm G(X)$ should be nonnegative when $X<0$. So~there are two cases:
\begin{enumerate}\itemsep=0pt
\item
If $\lambda=-{\rm i}^{-n}{\rm e}^{-\pi {\rm i}c}$ then $G(X)$ should be nonnegative when $X>0$ and nonpositive when $X<0$.
\item
If $\lambda=+{\rm i}^{-n}{\rm e}^{-\pi {\rm i}c}$ then $G(X)$ should be nonnegative for all $X\in \RN$.
\end{enumerate}
This proves the proposition.
\end{proof}

We deduce the following theorem from Propositions~\ref{PropFromPositiveTraceToFunctions} and~\ref{PropFromFunctionsToPolynomials}:
\begin{Theorem}
\label{ThrPositiveFormSmallerThanOne}
Suppose that $\mc{A}$ is a deformation of $A=\CN[p,q]^{\mathbb Z/n}$ with conjugation $\rho$ as above, $\rho^2=g_t$, $t=\exp(2\pi {\rm i} c)$. Let~$P(x)$ be the parameter of $\mc{A}$, $\varepsilon={\rm i}^n{\rm e}^{\pi {\rm i}c}\lambda=\pm 1$ $($so $\rho=\rho_\varepsilon)$. Then the cone $\mc C_+$ of positive definite $\rho$-invariant forms on $\mc{A}$ is isomorphic to the cone of nonzero polynomials $G(X)$ of degree $\le n-1$ with $G(0)=0$ if $c=0$ such that
\begin{enumerate}\itemsep=0pt
\item[$1.$] If $\varepsilon=-1$ then $G(X)$ is nonnegative when $X>0$ and nonpositive when $X<0$.
\item[$2.$] If $\varepsilon=1$ then $G(X)$ is nonnegative for all $X\in \RN$.
\end{enumerate}
\end{Theorem}

Thus for $\rho=\rho_-$, $G(X)=XU(X)$ where $U(X)\ge 0$ is a polynomial of degree $\le n-2$, and for $\rho=\rho_+$, $G(X)\ge 0$ is a polynomial of degree $\le n-1$ with $G(0)=0$ if $c=0$; in the latter case $G(X)=X^2U(X)$ where $U(X)\ge 0$ is a polynomial of degree $\le n-3$. Therefore, we get

\begin{Proposition} The dimension of $\mc C_+$ modulo scaling is
\begin{enumerate}\itemsep=0pt
\item[$\bullet$] $n-2$ for even $n$ and $n-3$ for odd $n$ if $\rho=\rho_-$;

\item[$\bullet$] $n-2$ for even $n$ and $n-1$ for odd $n$ if $c\ne 0$ and $\rho=\rho_+$;

\item[$\bullet$] $n-4$ for even $n$ and $n-3$ for odd $n$ if $c=0$ and $\rho=\rho_+$.
\end{enumerate}
$($Here if the dimension is $<0$, the cone $\mc C_+$ is empty$.)$
\end{Proposition}

Consider now the special case of even short star-products (i.e., quaternionic structures). Let~$\mc A$ be an even quantization of $A$, and
$\mc C_+^{\rm even}$ the cone of positive $\sigma$-stable $s$-twisted traces (i.e., those defining even short star-products). Then we have

\begin{Proposition}\label{eve}
The dimensions
of $\mc C_+^{\rm even}$ modulo scaling in various cases are as follows:
\begin{enumerate}\itemsep=0pt
\item[$\bullet$] $\frac{n-3}{2}$ if $\rho=\rho_-$, $n$ odd;

\item[$\bullet$] $\frac{n-1}{2}$ if $\rho=\rho_+$, $n$ odd;

\item[$\bullet$] $\frac{n-2}{2}$ if $\rho=\rho_-$, $n$ even;

\item[$\bullet$] $\frac{n-4}{2}$ if $\rho=\rho_+$, $n$ even.
\end{enumerate}
\end{Proposition}

Proposition~\ref{eve} shows that the only cases of a unique positive
$\sigma$-stable $s$-twisted trace are $\rho=\rho_+$ for $n=1,4$ and $\rho=\rho_-$ for $n=2,3$.

The paper~\cite{BPR} considers the case $\rho=\rho_+$ if $n=0,1$ mod $4$
and $\rho=\rho_-$ if $n=2,3$ mod $4$; this is the canonical quaternionic structure of the hyperK\"ahler cone (see~\cite[Section~3.8]{ES}), since it is obtained from $\rho_+$ on $\mathbb C[p,q]$ by restricting to $\mathbb Z/n$-invariants. Thus for $n\le 4$ the unitary even star-product is unique, as conjectured in~\cite{BPR}. However, for $n\ge 5$ this is no longer so. For~example, for $n=5$ (a case commented on at the end of section 6 of~\cite{BPR}) by Proposition~\ref{eve} the cone $\mc C_+^{\rm even}$ modulo scaling is 2-dimensional (which disproves the most optimistic conjecture of~\cite{BPR} that a unitary even star-product is always unique).\footnote{It is curious that in the case considered in~\cite{BPR}, the dimension of $\mc C_+^{\rm even}$ modulo scaling is always even.}

\begin{Example} Let $n=1$, $P(x)=x$, so $\mathbf{P}(X)=X+1$.
Then for $\rho=\rho_-$ there are no positive traces while for $\rho=\rho_+$
positive traces exist only if $c\ne 0$. In this case there is a unique positive trace up to scaling given by the weight function
\[
{w}(x)=\frac{{\rm e}^{2\pi {\rm i}\left(c-\frac{1}{2}\right)x}}{2\cos \pi x}.
\]
In particular, the only quaternionic case is $\rho=\rho_+$, $c=\frac{1}{2}$,
which gives ${w}(x)=\frac{1}{2\cos \pi x}$.

\end{Example}

\begin{Example}\label{neq2} Let $n=2$, $P(x)=x^2+\beta^2$, $\beta^2\in \mathbb R$ so we have $\mathbf{P}(X)=(X+{\rm e}^{2\pi\beta})(X+{\rm e}^{-2\pi\beta})$. We~assume that $\beta^2>-\frac{1}{4}$ so that all roots of $P$ are in the strip $|\Re x|<\frac12$. Then $\rho=\rho_-$ gives a unique up to scaling positive trace defined by the weight function
\[
{w}(x)=\frac{{\rm e}^{2\pi {\rm i}cx}}{4\cos \pi(x-\beta)\cos \pi(x+\beta)},
\]
and $\rho=\rho_+$ is possible if and only if $c\ne 0$ and gives a unique up to scaling positive trace defined by the weight function
\[
{ w}(x)=\frac{{\rm e}^{2\pi{\rm i}(c-1)x}}{4\cos \pi(x-\beta)\cos \pi(x+\beta)}.
\]

In particular, the only quaternionic case is $\rho=\rho_-$, $c=0$, with
\[
{w}(x)=\frac{1}{4\cos \pi(x-\beta)\cos \pi(x+\beta)},
\]
which corresponds to the ${\rm SL}_2$-invariant short star-product.
There are two subcases: $\beta^2\ge 0$, which corresponds to the
{\it spherical unitary principal series} for ${\rm SL}_2(\mathbb C)$, and
$-\frac{1}{4}<\beta^2<0$, which corresponds to the {\it spherical unitary complementary series} for the same group (namely, the trace form is exactly the positive inner product on the underlying Harish-Chandra bimodule).

Note that together with the trivial representation $\big($corresponding to $\beta^2=-\frac{1}{4}\big)$, these representations are well known to exhaust irreducible spherical unitary representations of ${\rm SL}_2(\mathbb C)$~\cite{V}.
\end{Example}

\begin{Example}
Let $n=3$ and $P(x)=x^3+\beta^2x=x(x-{\rm i}\beta )(x+{\rm i}\beta )$, where $\beta^2\in \mathbb R$.
This gives the algebra defined by formulas~(6.17), (6.18)
of~\cite{BPR}, with $\zeta=1$; namely,
the generators $\hat X$, $\hat Y$, $\hat Z$ of~\cite{BPR} are $v$, $u$, $z$, respectively, and the parameter $\kappa$ of~\cite{BPR} is $\kappa=-\beta^2-\frac{1}{4}$.
This is an even quantization of $A=\mathbb C[X_3]$. Thus even short star-products are parametrized by a~single parameter $\alpha$; namely, the corresponding $\sigma$-invariant $s$-twisted trace such that $T(1)=1$ is determined by the condition that $T(z^2)=-\alpha$ (using the notation of~\cite{BPR}).

Assume that $\beta^2>-\frac{1}{4}$ (i.e., $\kappa<0$), so that all the roots of $P$ are in the strip $|\Re x|<\frac{1}{2}$. We~have
\[
\mathbf{P}(X)=(X+1)\big(X+{\rm e}^{2\pi \beta}\big)\big(X+{\rm e}^{-2\pi \beta}\big).
\]
In this case $c=\frac{1}{2}$ so the trace $T$, up to scaling, is given by
\[
T(R(z))=\int_{{\rm i}\mathbb R}R(x)w(x)|{\rm d}x|,
\]
where
\[
w(x)={\rm e}^{\pi {\rm i}x}\frac{G({\rm e}^{2\pi {\rm i}x})}{({\rm e}^{2\pi {\rm i}x}+1)({\rm e}^{2\pi {\rm i}(x-{\rm i}\beta )}+1)({\rm e}^{2\pi {\rm i}(x+{\rm i}\beta )}+1)},
\]
with $\deg(G)\le 2$. Moreover, because of evenness we must have $w(x)=w(-x)$, so $G(X)=X^2G\big(X^{-1}\big)$. Up to scaling, such polynomials $G$ form a 1-parameter family, parametrized by $\alpha$.

Following~\cite[Section~6.3]{BPR}, let us equip the corresponding quantum algebra $\mc A=\mc A_P$ with the quaternionic structure $\rho_-$ given by\footnote{Note that our $\rho$ is $\rho^{-1}$ in~\cite{BPR}, so we use $\rho_-$ while~\cite{BPR} use $\rho_+=\rho_-^{-1}$.}
\[
\rho_-(v)=-u,\qquad \rho_-(u)=v,\qquad \rho_-(z)=-z,
\]
 and let us determine which traces are unitary
for this quaternionic structure. According to The\-o\-rem~\ref{ThrPositiveFormSmallerThanOne}, there is a unique
such trace (which is automatically $\sigma$-stable), corresponding to~$G(X)=X$. Thus this trace is given by the weight function
\[
w(x)=\frac{1}{\cos \pi x\cos \pi (x-{\rm i}\beta )\cos\pi(x+{\rm i}\beta )}.
\]
Hence,
\[
T(z^{k})=\int_{{\rm i}\mathbb R}\frac{x^{k}|{\rm d}x|}{\cos \pi x\cos \pi (x-{\rm i}\beta )\cos\pi(x+{\rm i}\beta )},
\]
in particular, $T(z^k)=0$ if $k$ is odd.

For even $k$ this integral can be computed using the residue formula. Namely, assume $\beta\ne 0$ and let us first compute $T(1)$. Replacing the contour ${\rm i}\mathbb R$ by $1+{\rm i}\mathbb R$ and subtracting,
we find using the residue formula:
\[
2T(1)=2\pi\big({\rm Res}_{\frac{1}{2}}w+{\rm Res}_{\frac{1}{2}-{\rm i}\beta }w+{\rm Res}_{\frac{1}{2}+{\rm i}\beta }w\big).
\]
Now,
\[
{\rm Res}_{\frac{1}{2}}w=\frac{1}{\pi\sinh^2\pi \beta},
\]
while
\[
{\rm Res}_{\frac{1}{2}-{\rm i}\beta }w={\rm Res}_{\frac{1}{2}+{\rm i}\beta }w=-\frac{1}{\pi \sinh \pi\beta \sinh 2\pi \beta}.
\]
Thus
\[
T(1)=\frac{1}{\sinh^2\pi \beta}-\frac{2}{\sinh \pi\beta \sinh 2\pi\beta }=
\frac{1}{\sinh^2\pi \beta}\bigg(1-\frac{1}{\cosh \pi\beta }\bigg)=\frac{1}{2\cosh^2(\frac{\pi\beta }{2})\cosh \pi\beta }.
\]
Note that this function has a finite value at $\beta=0$, which is the answer in that case.

Now let us compute $T\big(z^2\big)$.
Again replacing the contour ${\rm i}\mathbb R$ with $1+{\rm i}\mathbb R$ and subtracting,
we~get
\[
T(1)+2T\big(z^2\big)=T\big(z^2\big)+T\big((z+1)^2\big)=2\pi\big({\rm Res}_{\frac{1}{2}}x^2w+{\rm Res}_{\frac{1}{2}-{\rm i}\beta }x^2w+{\rm Res}_{\frac{1}{2}+{\rm i}\beta }x^2w\big).
\]
Now,
\[
{\rm Res}_{\frac{1}{2}}x^2w=\frac{1}{4\pi\sinh^2\pi \beta},
\]
while
\[
{\rm Res}_{\frac{1}{2}-{\rm i}\beta }x^2w+{\rm Res}_{\frac{1}{2}+{\rm i}\beta }x^2w=\frac{2\beta^2-\frac{1}{2}}{\pi \sinh \pi\beta \sinh 2\pi \beta}.
\]
So
\[
T\big(z^2\big)=-\frac{1}{4\sinh^2\pi \beta}+\frac{2\beta^2+\frac{1}{2}}{\sinh \pi\beta \sinh 2\pi \beta}=
\frac{1}{\sinh^2\pi \beta}\bigg(-\frac{1}{4}+\frac{\beta^2+\frac{1}{4}}{\cosh \pi \beta}\bigg).
\]
Thus,
\[
\alpha=-\frac{T\big(z^2\big)}{T(1)}=\frac{1}{4}+\frac{\beta^2}{1-\cosh \pi\beta }=\frac{1}{4}-\frac{\kappa+\tfrac{1}{4}}{1-\cos\pi\sqrt{\kappa+\tfrac{1}{4}}}.
\]
This gives the equation of the curve in Fig.~2 in~\cite{BPR}. We~also note that for $\kappa=-\frac{1}{4}$ (i.e., $\beta=0$) we get
$\alpha=\frac{1}{4}-\frac{2}{\pi^2}$.
the value found in~\cite{BPR}.
\end{Example}

\begin{Example}
Let $n=4$ and
\[
P(x)=\big(x^2+\beta^2\big)\big(x^2+\gamma^2\big)=(x-{\rm i}\beta )(x+{\rm i}\beta )(x-{\rm i}\gamma)(x+{\rm i}\gamma),
\]
where $\beta^2,\gamma^2\in \mathbb R$.
This is an even quantization of $A=\mathbb C[X_4]$ discussed in~\cite[Section~6.4]{BPR}. Thus even short star-products are still parametrized by a single parameter $\alpha$; namely, the corresponding $\sigma$-invariant $s$-twisted trace such that $T(1)=1$
is determined by the condition that $T\big(z^2\big)=-\alpha$.

Assume that $\beta^2,\gamma^2>-\frac{1}{4}$, so that all the roots of $P$ are in the strip $|\Re x|<\frac{1}{2}$. We~have
\[
\mathbf{P}(X)=\big(X+{\rm e}^{2\pi \beta}\big)\big(X+{\rm e}^{-2\pi \beta}\big)\big(X+{\rm e}^{2\pi \gamma}\big)\big(X+{\rm e}^{-2\pi \gamma}\big).
\]
In this case $c=0$ so the trace $T$, up to scaling, is given by
\[
T(R(z))=\int_{{\rm i}\mathbb R}R(x)w(x)|{\rm d}x|,
\]
where
\[
w(x)=\frac{G({\rm e}^{2\pi {\rm i}x})}{({\rm e}^{2\pi {\rm i}(x-{\rm i}\beta )}+1)({\rm e}^{2\pi {\rm i}(x+{\rm i}\beta )}+1)({\rm e}^{2\pi {\rm i}(x-{\rm i}\gamma)}+1)({\rm e}^{2\pi {\rm i}(x+{\rm i}\gamma)}+1)},
\]
with $\deg(G)\le 3$ and $G(0)=0$. Moreover, because of evenness we must have \mbox{$w(x)=w(-x)$}, so~$G(X)=X^4G\big(X^{-1}\big)$. Up to scaling, such polynomials $G$ form a 1-parameter family, para\-met\-ri\-zed by $\alpha$.

Let us equip the corresponding quantum algebra $\mc A=\mc A_P$ with the quaternionic structure $\rho_+$ given by $\rho_+(v)=u$, $\rho_+(u)=v$, $\rho_+(z)=-z$, and let us determine which traces are unitary
for this quaternionic structure. According to Theorem~\ref{ThrPositiveFormSmallerThanOne}, there is a unique
such trace (which is automatically $\sigma$-stable), corresponding to $G(X)=X^2$. Thus this trace is given by the weight function
\[
w(x)=\frac{1}{\cos \pi (x-{\rm i}\beta )\cos\pi(x+{\rm i}\beta )\cos \pi (x-{\rm i}\gamma)\cos\pi(x+{\rm i}\gamma)}.
\]
Thus,
\[
T\big(z^{k}\big)=\int_{{\rm i}\mathbb R}\frac{x^{k}|{\rm d}x|}{\cos \pi (x-{\rm i}\beta )\cos\pi(x+{\rm i}\beta )\cos \pi (x-{\rm i}\gamma)\cos\pi(x+{\rm i}\gamma)},
\]
in particular, $T\big(z^k\big)=0$ if $k$ is odd.

As before, for even $k$ this integral can be computed using the residue formula. Namely, assume $\beta\ne 0$, $\gamma\ne 0$, $\beta\ne \pm \gamma$, and let us first compute $T(1)$. Replacing the contour ${\rm i}\mathbb R$ by $1+{\rm i}\mathbb R$ and subtracting, we find using the residue formula:
\begin{gather*}
T\big((z+1)^2\big)-T\big(z^2\big)=T(1)
\\ \hphantom{T\big((z+1)^2\big)-T\big(z^2\big)}
{}=-2\pi\big({\rm Res}_{\frac{1}{2}-{\rm i}\beta }x^2w+{\rm Res}_{\frac{1}{2}+{\rm i}\beta }x^2w+{\rm Res}_{\frac{1}{2}-{\rm i}\gamma}x^2w+{\rm Res}_{\frac{1}{2}+{\rm i}\gamma}x^2w\big).
\end{gather*}
Now,
\[
{\rm Res}_{\frac{1}{2}-{\rm i}\beta }x^2w+{\rm Res}_{\frac{1}{2}+{\rm i}\beta }x^2w=\frac{2\beta}{\pi \sinh \pi(\beta+\gamma)\sinh\pi(\gamma-\beta)\sinh 2\pi \beta}.
\]
Thus
\[
T(1)=\frac{1}{\sinh \pi(\beta+\gamma)\sinh\pi(\gamma-\beta)}\bigg(\frac{4\beta}{\sinh 2\pi \beta}-\frac{4\gamma}{\sinh 2\pi \gamma}\bigg).
\]
Note that this function is regular when $\beta\gamma(\beta-\gamma)(\beta+\gamma)=0$, and the corresponding limit is the answer in that case.

We similarly have
\begin{gather*}
T\big((z+1)^4\big)-T\big(z^4\big)=6T\big(z^2\big)+T(1)
\\ \hphantom{T\big((z+1)^4\big)-T\big(z^4\big)}
{}=-2\pi\big({\rm Res}_{\frac{1}{2}-{\rm i}\beta }x^4w+{\rm Res}_{\frac{1}{2}+{\rm i}\beta }x^4w+{\rm Res}_{\frac{1}{2}-i\gamma}x^4w+{\rm Res}_{\frac{1}{2}+i\gamma}x^4w\big),
\end{gather*}
and
\[
{\rm Res}_{\frac{1}{2}-{\rm i}\beta }x^4w+{\rm Res}_{\frac{1}{2}+{\rm i}\beta }x^4w=\frac{\beta-4\beta^3}{\pi \sinh \pi(\beta+\gamma)\sinh\pi(\gamma-\beta)\sinh 2\pi \beta}.
\]
Thus
\[
6T\big(z^2\big)+T(1)=\frac{2}{\sinh \pi(\beta+\gamma)\sinh\pi(\gamma-\beta)}\bigg(\frac{\beta-4\beta^3}{\sinh 2\pi \beta}-\frac{\gamma-4\gamma^3}{\sinh 2\pi \gamma}\bigg).
\]
Hence
\[
T\big(z^2\big)=\frac{1}{\sinh \pi(\beta+\gamma)\sinh\pi(\gamma-\beta)}\bigg(\frac{\gamma+4\gamma^3}{3\sinh 2\pi \gamma}-\frac{\beta+4\beta^3}{3\sinh 2\pi \beta}\bigg).
\]
Thus
\[
\alpha=-\frac{T\big(z^2\big)}{T(1)}=\frac{1}{12}+\frac{1}{3}\frac{\beta^3\sinh 2\pi \gamma-\gamma^3\sinh 2\pi \beta}{\beta\sinh 2\pi \gamma-\gamma\sinh 2\pi \beta}.
\]
This is the equation (in appropriate coordinates) of the surface computed numerically in~\cite{BPR} and shown in Fig.~4 of that paper.
In particular, for $\beta=\gamma=0$, we get
\[
\alpha=\frac{1}{12}-\frac{1}{2\pi^2}.
\]
Thus $\tau=128\alpha=\frac{32(\pi^2-6)}{3\pi^2}=4.18211{\dots}$
is the number given by the complicated expres\-sion~(B.16) of~\cite{BPR}
(as was pointed out in~\cite{DPY}).
\end{Example}

\begin{Remark} Similar calculations can be found in~\cite[Section~8.1]{DPY}.
\end{Remark}

\subsection{The case of a closed strip}
Suppose now that all roots $\alpha$ of $P$ satisfy $|\Re\alpha|\leq \frac{1}{2}$. Recall that we have $P(x)=P_*(x)Q\big(x+\frac{1}{2}\big)$ $\times Q\big(x-\tfrac{1}{2}\big)$ where all roots of $P_*(x)$ satisfy $|\Re x|<\frac{1}{2}$ and all roots of $Q(x)$ belong to ${\rm i}\RN$. For~any $R\in \CN[x]$ write $R=R_1Q+R_0$, where $\deg R_0<\deg Q$.

By Proposition~\ref{PropTracesForClosedStrip} each $g_t$-twisted trace can be obtained as
\[
T(R(z))=\int_{{\rm i}\RN}R_1(x)Q(x)w(x)|{\rm d}x|+\phi(R_0),
\]
where $w(x)={\rm e}^{2\pi {\rm i} cx}\frac{G({\rm e}^{2\pi {\rm i} x})}{\bP({\rm e}^{2\pi {\rm i}x})}$ and $\phi$ is any linear functional.

\begin{Proposition}
\label{PropPositiveFormNoPoles}
Suppose that $T$ is a trace as above and $w(x)$ has poles on ${\rm i}\RN$. Then $T$ does not give a positive definite form.
\end{Proposition}
\begin{proof}
Let $Q_*(x)=Q(x)\ovl{Q}(-x)$; note that $Q_*(x)\ge 0$ for $x\in {\rm i}\mathbb R$. Then there exists a linear functional $\psi$ such that for any $R=R_1Q_*+R_0$ with $\deg R_0<\deg Q_*$ we have
\[
T(R(z))=\int_{{\rm i}\RN}R_1(x)Q_*(x)w(x)|{\rm d}x|+\psi(R_0).
\]

Suppose that $T$ gives a positive definite form. Then $T(S(z)\ovl{S}(-z))>0$ for all nonzero \mbox{$S\in\CN[x]$}. Taking $S(x)=Q_*(x)S_1(x)$ and using Lemma~\ref{CorClassical}, we deduce that $Q_*^2(x)w(x)$, hence~$w(x)$, is nonnegative on ${\rm i}\RN$. In particular, all poles of $w(x)$ have order at least $2$.

Without loss of generality assume that $w(x)$ has a pole at zero. Let
\[
R_n(x):=(F_nQ_*+b)\big(\ovl{F_n}Q_*+b\big),
\]
where $b\in \RN$. Suppose that $F_n$ is a sequence of polynomials that tends to the function $f:=\chi_{(-\eps,\eps)}$ (the characteristic function of the interval) in the space $L^{2}\big({\rm i}\RN, (Q_*+Q_*^2\big)w)$. In particular, $F_n$~tends to $f$ in the spaces $L^2({\rm i}\RN,Q_* w)$ and $L^2\big({\rm i}\RN,Q_*^2 w\big)$. Then we deduce from the Cauchy--Schwartz inequality that $F_n\ovl{F_n}$ tends to $f^2$ in the space $L^1\big({\rm i}\RN,Q_*^2w\big)$, and $F_n$ and $\ovl{F_n}$ tend to $f$ in $L^1({\rm i}\RN,Q_*w)$.

We have
\begin{gather*}
T(R_n(z))=T\big(\big(F_n\ovl{F_n}Q_*+F_nb+\ovl{F_n}b\big)(z)Q_*(z)+b^2\big)
\\ \hphantom{T(R_n(z))}
{}=\int_{{\rm i}\RN}\big(F_n\ovl{F_n}Q_*^2+F_nbQ_*+\ovl{F_n}bQ_*\big)w|{\rm d}x|+\phi\big(b^2\big).
\end{gather*}
Therefore, when $n$ tends to infinity,
\[
T(R_n(z))\to \int_{{\rm i}\RN}\big(f^2Q_*^2+2fbQ_*\big)w|{\rm d}x|+\phi\big(b^2\big).
\]
We have $\phi\big(b^2\big)=Cb^2$ for some $C\geq 0$. Suppose that $w$ has a pole of order $M\geq 2$ at $0$ and $Q_*$ has a zero of order $N>0$ at $0$. Then $Q_*w$ has a zero of order $N-M$ at zero and $Q_*^2w$ has a~zero of order $2N-M$ at zero. We~deduce that
\[
\int_{{\rm i}\RN}F_nQ_*w|{\rm d}x|\to c_1\eps^{N-M+1},\qquad \int_{{\rm i}\RN}F_nQ_*^2w|{\rm d}x|\to c_2\eps^{2N-M+1},\qquad n\to \infty,
\]
 where $c_1=c_1(\varepsilon)$, $c_2=c_2(\varepsilon)$ are functions having strictly positive limits at $\varepsilon=0$ . Therefore
 \[
 \lim_{n\to\infty}T(R_n(z))=Cb^2+2c_1\eps^{N-M+1}b+c_2\eps^{2N-M+1}.
 \]
 This is a quadratic polynomial of $b$ with discriminant
 \[
 D=4\eps^{2N-2M+2}\big(c_1^2-Cc_2\eps^{M-1}\big).
 \]
 Since $M{\geq }2$, for small $\eps$ this discriminant is positive. In particular, for some $b$,
 $ \lim\limits_{n\to\infty}T(R_n(z)){<}0,$
 so for this $b$ and some $n$, $T(R_n(z))<0$, a contradiction.
\end{proof}

Now we are left with the case when $w(x)$ has no poles on ${\rm i}\RN$. In this case $T(R(z))=\int_{{\rm i}\RN}R(x)w(x)|{\rm d}x|+\eta(R_0)$, where $\eta$ is some linear functional.
\begin{Proposition}
\label{PropIfPositiveThenNoDerivatives}
$T$ gives a positive definite form only when $\eta(R_0)=\sum_{j}c_j R_0(z_j)$, where $c_j\geq 0$ and $z_j\in {\rm i}\mathbb R$ are the roots of $Q$.
\end{Proposition}
\begin{proof}
Suppose that this is not the case. Then it is easy to find a polynomial $S$ such that $\eta\big(\big(S\ovl{S}\big)_0\big)<0$. Then using Lemma~\ref{CorClassical}(2) for $M=Q$, we find $F_n$ such that $F_nQ+S$ tends to zero in $L^2({\rm i}\RN,w)$. We~deduce that
\[
T\big((F_nQ+S)(z)\ovl{(F_nQ+S)}(z)\big)\to \eta\big((S\ovl{S})_0\big)<0,
\]
which gives a contradiction.
\end{proof}

In the proof of Proposition~\ref{PropPositiveFormNoPoles} we got that $w$ is nonnegative on ${\rm i}\RN$. We~also note that $Q(z)$ divides $P\big(z-\tfrac12\big)$, hence
\[
T\big(R\big(z-\tfrac12\big)\ovl{R}\big(\tfrac12-z\big)P\big(z-\tfrac12\big)\big)
=\int_{{\rm i}\RN}R\big(z-\tfrac12\big)\ovl{R}\big(\tfrac12-z\big)P\big(z-\tfrac12\big) w(z) |dz|.
\]
Using the proof of Proposition~\ref{PropFromPositiveTraceToFunctions} we see that $\lambda P(z)w\big(z+\tfrac12\big)$ is nonnegative on ${\rm i}\RN$. Assume that $Q\big(z-\tfrac12\big)Q\big(z+\tfrac12\big)$ is positive on $\RN$. Then this is equivalent to $\lambda P_*(z)w\big(z+\tfrac12\big)$ being nonnegative on ${\rm i}\RN$. So~we have proved the following theorem.

\begin{Theorem}
\label{ThrPositiveDefiniteFormsClosedStrip}
Suppose that $P(x)=P_*(x)Q\big(x-\tfrac{1}{2}\big)Q\big(x+\frac{1}{2}\big)$, where all roots of $P_*$ belong to the set $|\Re x|<\frac{1}{2}$ and all roots of $Q$ belong to ${\rm i}\RN$. Suppose that $\alpha_1,\ldots,\alpha_k$ are all the different roots of $Q$. Then positive traces $T$ are in one-to-one correspondence with $\widetilde{T}$, $c_1,\ldots, c_k\geq 0$, where $\widetilde{T}$ is a positive trace for $P_*$; namely, \[T(R(z))=\widetilde{T}(R(z))+\sum c_i R(\alpha_i).\]
\end{Theorem}

\subsection{The general case}
Let $\mc{A}$ is be a filtered quantization of $A$ with conjugation $\rho$ given by
formula \eqref{rho}. Let~$P(x)$ be its parameter. Let~$\widetilde{P}(x)$ be the polynomial defined in Section~\ref{SubSubSecGeneralTraces}: it has the same degree as~$P(x)$ and its roots are obtained from the roots of $P(x)$ by minimal integer shift into the strip $|\Re x|\leq \frac{1}{2}$.
Also recall from Section~\ref{SubSubSecGeneralTraces}
that $P_\circ$ denotes the following polynomial:
all roots of~$P_\circ$ belong to strip $|\Re x|\leq \frac{1}{2}$ and the multiplicity of $\alpha, |\Re\alpha|\leq \frac{1}{2}$ in $P_\circ$ equals to multiplicity of~$\alpha$ in~$P$. Let~$n_\circ:=\deg(P_\circ)$.

Proposition~\ref{PropGeneralTrace} says that any trace $T$ can be represented as $T=\Phi+\widetilde{T}$, where $\Phi$ is a linear functional such that
\[
\Phi(R)=\sum_{j=1}^m\sum_k c_{jk} R^{(k)}(z_j),
\]
 $z_j\notin {\rm i}\RN$, and $\widetilde{T}$ is a trace for $\widetilde{P}$. Furthermore, if $\Phi=0$ then $T$ is just a trace for $P_\circ$.

\begin{Proposition}
\label{PropPositivePhiIsZero}
Let $T$ be a trace such that $\Phi\neq 0$. Then $T$ does not give a positive definite form.
\end{Proposition}
\begin{proof}
For big enough $k$ we have $\Phi((x-z_1)^k\cdots (x-z_m)^k \CN[x])=0$. Recall that there exists polynomial $Q_*(x)$ nonnegative on ${\rm i}\RN$ such that for $R=R_1Q_*+R_0$, $\deg R_0<\deg Q_*$, we have $\widetilde{T}(R)=\int_{{\rm i}\RN} R_1Q_* w |{\rm d}x|+\psi(R_0)$, where $\psi$ is some linear functional. Let~$U(x)$ be a polynomial divisible by $Q_*$ such that $\Phi(U(x)\CN[x])=0$.

Let $L$ be any polynomial. Using Lemma~\ref{CorClassical} for $M=U$, we find a sequence $G_n=US_n$ that tends to $L$ in the space $L^2({\rm i}\RN, Q_*w|{\rm d}x|)$. We~deduce that $H_n(x):=(G_n(x)-L(x))(\ovl{G_n}(-x)-\ovl{L}(-x))$ tends to zero in $L^1({\rm i}\RN,Q_*w)$. We~have $\widetilde{T}(H_n(z)Q_*(z))=\int_{{\rm i}\RN}H_n(x) Q_* w|{\rm d}x|$. We~conclude that $\widetilde{T}(H_n(z)Q_*(z))=\|H_n\|_{L^1({\rm i}\RN,Q_*w)}$ tends to zero when $n$ tends to infinity.

It follows that $T(H_n(z))$ tends to $\Phi(Q_*(x)H_n(x))=\Phi\big(Q_*(x)L(x)\ovl{L}(-x)\big)$. Since $H_n$ is nonnegative on ${\rm i}\RN$, we have $T(H_n(z))>0$. Now we get a contradiction with
\begin{Lemma}
There exists $F(x)\in\CN[x]$ such that $\Phi\big(Q_*(x)F(x)\ovl{F}(-x)\big)<0$.
\end{Lemma}
\begin{proof}
Let $r$ be the biggest number such that there exists $j$ with $c_{jr}\neq 0$. Let
\[
F(x):=G(x)(x-z_1)^{r+1}\cdots(x-z_j)^{r}\cdots (x-z_m)^{r+1}.
\]
Here we omit $x-z_{{j^*}}$ in the product, where ${j^*}\ne j$ is such that $z_{{j^*}}=-\ovl{z_j}$. We~note that $c_{ik} \big(Q_*(x)F(x)\ovl{F}(-x)\big)^{(k)}(z_i)=0$ for all $i,k$ except $k=r$ and $i=j$ or $i={j^*}$. It follows that
\begin{gather*}
\Phi\big(Q_*(x)F(x)\ovl{F}(-x)\big)=
c_{jr}\big(Q_*(x)F(x)\ovl{F}(-x)\big)^{(r)}(z_j)+c_{{j^*}r} \big(Q_*(x)F(x)\ovl{F}(-x)\big)^{(r)}(z_{{j^*}})
\\ \hphantom{\Phi\big(Q_*(x)F(x)\ovl{F}(-x)\big)}
{}=c_{jr} Q_*(z_j)F^{(r)}(z_j)\ovl{F}(-z_j)+(-1)^r c_{{j^*}r}Q_*(z_{{j^*}})F(z_{{j^*}})\ovl{F}^{(r)}(-z_{{j^*}})
\\ \hphantom{\Phi\big(Q_*(x)F(x)\ovl{F}(-x)\big)}
{}=c_{jr}a+c_{{j^*}r}\ovl a,
\end{gather*}
where $a:= Q_*(z_{j})F^{(r)}(z_{j})\ovl{F}(-z_{j})$. Pick $a\in \mathbb C$ so that
\[
c_{jr}a+c_{{j^*}r}\ovl a=2{\rm Re}(c_{jr}a)<0,
\]
and choose
$G\in \CN[x]$ which gives this value of $a$ (e.g., we can choose $G$ to be linear).
Then $\Phi\big(Q_*(x)F(x)\ovl{F}(-x)\big)<0$, as desired.
\end{proof}
\renewcommand{\qed}{}
\end{proof}

If $\mc{A}_{P_\circ}$ is the quantization with parameter $P_\circ$ then there is a conjugation $\rho_\circ$ on $\mc{A}_{P_\circ}$ given by the formulas
\[
\rho_\circ(v)=\lambda_\circ u,\qquad \rho_\circ(u)=(-1)^n \lambda_\circ^{-1}v,\qquad \rho_\circ(z)=-z,
\]
where $\lambda_\circ:=(-1)^{\frac{n-n_\circ}{2}}\lambda$. Therefore we can consider the cone of positive definite forms for $\mc{A}_{P_\circ}$ with respect to $\rho_\circ$.
\begin{Corollary}
The cone of positive definite forms on $\mc{A}_P$ with respect to $\rho$ coincides with the cone of positive definite forms on $\mc{A}_{P_\circ}$ with respect to $\rho_\circ$. Namely, a trace $T\colon \CN[x]\to \CN$ for $\mc{A}$ gives a positive definite form if and only if $T$ is a trace for $\mc{A}_{P_\circ}$ that gives a positive definite form on $\mc{A}_{P_\circ}$.
\end{Corollary}
\begin{proof}
We deduce from Proposition~\ref{PropPositivePhiIsZero} that each trace $T$ that gives a positive definite form should have $\Phi=0$. By Proposition~\ref{PropGeneralTrace}, in this case $T$ is a trace for the polynomial $P_\circ(x)$. So~there exists polynomial $Q_*$ such that for $R=R_1Q_*+R_0$, $\deg R_0<\deg Q_*$, and
\[
T(R(z))=\int_{{\rm i}\RN} Q_*(x)R_1(x)w(x) |{\rm d}x|+\phi(R_0).
\]

Using Proposition~\ref{PropPositiveFormNoPoles} and its proof, we deduce that $w$ has no poles and that $w(x)$ and $\lambda w\big(x+\frac{1}{2}\big)P(x)$ are nonnegative on ${\rm i}\RN$. Therefore
\[
T(R(z))=\int_{{\rm i}\RN} R(x) w(x) |{\rm d}x|+\psi(R_0),
\]
where $\psi$ is some linear functional. Using Proposition~\ref{PropIfPositiveThenNoDerivatives}, we deduce that this trace is positive if and only if $\psi(R_0)=\sum_j c_j R_0(z_j)$, where $c_j\geq 0$ and $z_j\in {\rm i}\mathbb R$.

Since $(-1)^{\frac{n-n_\circ}{2}}\frac{P(x)}{P_\circ(x)}$ is positive on ${\rm i}\RN$, we see that $\lambda w\big(x+\frac{1}{2}\big)P(x)$ is nonnegative on ${\rm i}\RN$ if and only if $\lambda_\circ w\big(x+\frac{1}{2}\big)P_\circ(x)$ is nonnegative on ${\rm i}\RN$. Using Theorem~\ref{ThrPositiveDefiniteFormsClosedStrip} we then deduce that $T$ is positive for $P(x)$ if and only if it is positive for $P_\circ(x)$.
\end{proof}
So we have proved the following theorem.

\begin{Theorem}
Let $\mc{A}=\mc A_P$ be a filtered quantization of $A$ with parameter $P$ equipped with a~con\-jugation $\rho$ such that $\rho^2=g_t$. Let~$\ell$ be the number of roots $\alpha$ of $P$ such that $|\Re\alpha|<\frac{1}{2}$ counted with multiplicities and $r$ be the number of distinct roots $\alpha$ of $P$ with $\Re\alpha=-\frac{1}{2}$. Then the cone $\mc C_+$ of $\rho$-equivariant positive definite traces on $\mc{A}$ is isomorphic to $\mc C_+^1\times \mc C_+^2$, where $\mc C_+^2=\RN_{\geq 0}^r$, and $\mc C_+^1$ is the cone of nonzero polynomials $G$ such that
\begin{enumerate}\itemsep=0pt
\item[$(1)$]
$G$ has degree less than $\ell$;
\item[$(2)$]
$G(0)=0$ if $t=1$;
\item[$(3)$]
$G(X)\geq 0$ when $X>0$;
\item[$(4)$]
$G(X)$ is either nonnegative or nonpositive when $X<0$ depending on whether
$\rho=\rho_+$ or~$\rho_-$.
\end{enumerate}
The conditions are the same as in Theorem~$\ref{ThrPositiveFormSmallerThanOne}$.
\end{Theorem}

\section[Explicit computation of the coefficients ak, bk of the 3-term recurrence for orthogonal polynomials and discrete Painleve systems]
{Explicit computation of the coefficients $\boldsymbol{a_k}$, $\boldsymbol{b_k}$ \\of the 3-term recurrence for orthogonal polynomials \\and discrete Painlev\'e systems}
\label{expcom}

As noted in Section~\ref{ortpol}, to compute
the short star-product associated to a trace $T$,
one needs to compute the coefficients $a_k$, $b_k$ of the 3-term recurrence for the corresponding orthogonal polynomials:
\[
p_{k+1}(x)=(x-b_k)p_k(x)-a_kp_{k-1}(x).
\]
Also recall \cite{Sz} that
$a_k=\frac{\nu_k}{\nu_{k-1}}$, where $\nu_k:=(P_k,P_k)$. Finally, recall that
$\nu_k=\frac{D_k}{D_{k-1}}$, where
$D_k$ is the Gram determinant for $1,x,\dots,x^{k-1}$, i.e.,
\[
D_k=\det_{0\le i,\,j\le k-1}\big(x^i,x^j\big)=\det_{0\le i,\,j\le k-1}(M_{i+j}),
\]
where $M_r$ is the $r$-th moment of the weight function $w(x)$, i.e.,
\[
M_r=\int_{{\rm i}\mathbb R}x^rw(x)|{\rm d}x|.
\]
In the even case $w(-x)=w(x)$ we have $b_k=0$, so
\[
p_{k+1}(x)=xp_k(x)-a_kp_{k-1}(x),
\]
and $p_k$ can be easily computed recursively from the sequence $a_k$.
If the polynomials $p_k$ are $q$-hypergeometric (i.e., obtained by a limiting procedure from Askey--Wilson polynomials), then $D_k$, $\nu_k$, $a_k$ admit explicit product formulas, but in general they do not admit any closed expression and do not enjoy any nice algebraic properties beyond the above.

In our case, the hypergeometric case only arises for $n=1$ or, in special
cases, $n=2$, but the fact that the weight function for general $n$ is
essentially a higher complexity version of the weight function for $n=1$
suggests that there is still a weaker algebraic structure in the
picture.

In fact, by~\cite{Magnus} it follows immediately from the fact that the
formal Stieltjes transform satisfies an inhomogeneous first-order
difference equation with rational coefficients that the corresponding orthogonal
polynomials $p_m(x)$ in the $x$-variable satisfy a family of difference equations
\[
\begin{pmatrix}
p_m\big(x+\frac{1}{2}\big)\\[1ex]
p_{m-1}\big(x+\frac{1}{2}\big)
\end{pmatrix}
= A_m(x)
\begin{pmatrix}
p_m\big(x-\tfrac{1}{2}\big)\\[1ex]
p_{m-1}\big(x-\tfrac{1}{2}\big)
\end{pmatrix}
\]
such that the matrix $A_m(x)$ has rational function coefficients of degree
bounded by a linear function of $n$ alone. (Here we work with the ``$x$''
version of the polynomials, to avoid unnecessary appearances of~${\rm i}$.)

Since the results of~\cite{Magnus} are stated in significantly more
generality than we need, we sketch how they apply in our special case. Let~$Y_0$ be the matrix
\[
Y_0(x) = \begin{pmatrix} 1 & F(x) \\ 0 & 1\end{pmatrix} ,
\]
where $F$ is the formal Stieltjes transform of the given trace. Moreover,
for each~$n$, let $\frac{q_n(x)}{p_n(x)}$ be the $n$-th Pad\'e approximant to $F(x)$
(with monic denominator), so that $\frac{q_n(x)}{p_n(x)}-F(x)=O\big(x^{-2n-1}\big)$. If we
define
\[
Y_n(x):= \begin{pmatrix} p_n(x) & -q_n(x)\\ p_{n-1}(x) & -q_{n-1}(x)\end{pmatrix}
Y_0(x)
\]
for $n>0$, then
\[
Y_n = \begin{pmatrix} x^n+o\big(x^n\big) & O\big(x^{-n-1}\big)\\
 x^{n-1}+o\big(x^{n-1}\big) & O\big(x^{-n}\big)\end{pmatrix} .
\]

\begin{Lemma}
 The denominator $p_n$ of the $n$-th Pad\'e approximant to $F(x)$ is the
 degree $n$ monic orthogonal polynomial for the associated linear
 functional $T$.
\end{Lemma}

\begin{proof}
 If $F=F_T$, then we find
 \[
 p_n(x) F(x) = T\bigg(\frac{p_n(x)}{x-z}\bigg) = T\bigg(\frac{p_n(x)-p_n(z)}{x-z}\bigg) +
 T\bigg(\frac{p_n(z)}{x-z}\bigg)
 \]
 (where we evaluate $T$ on functions of $z$, and $x$ is a parameter).
 The two terms correspond to the splitting of $p_n(x)F(x)$ into its
 polynomial part and its part vanishing at $x=\infty$, so that
 \[
 q_n(x) = T\bigg(\frac{p_n(x)-p_n(z)}{x-z}\bigg)
 \]
 and
 \[
 T\bigg(\frac{p_n(z)}{x-z}\bigg) = p_n(x)F(x)-q_n(x) = O\big(x^{-n-1}\big).
 \]
 Comparing coefficients of $x^{-m-1}$ for $0\le m<n$ implies that $T(z^m
 p_n(z))=0$ as required.
\end{proof}

{\samepage
\begin{Remark}\qquad
\begin{enumerate}\itemsep=0pt
\item[1.] It also follows that
 \[
 Y_n(x)_{12}=N_n x^{-n-1}+O\big(x^{-n-2}\big),\qquad Y_n(x)_{22}=N_{n-1} x^{-n}+O\big(x^{-n-1}\big).
 \]

\item[2.] Note that this is an algebraic/asymptotic version of the explicit solution of~\cite{BI} to the Riemann--Hilbert problem for orthogonal polynomials introduced in~\cite{FIK}.
\end{enumerate}
\end{Remark}

}

\begin{Lemma}
 We have $\det(Y_n)=N_{n-1}$ for all $n>0$.
\end{Lemma}

\begin{proof}
 The definition of $Y_n$ implies that $\det(Y_n)\in \CN[x]$, while the
 (formal) asymptotic behavior implies that $\det(Y_n)=N_{n-1}+O\big(\frac{1}{x}\big)$.
\end{proof}

The inhomogeneous difference equation satisfied by $F$ trivially induces an
inhomogeneous difference equation satisfied by~$Y_0$:
\[
Y_0\big(x+\tfrac{1}{2}\big)
=
\begin{pmatrix}
 1 & t^{-1}\frac{L(x)}{P(x)}\\[1ex]
 0 & t^{-1}
\end{pmatrix}
Y_0\big(x-\tfrac{1}{2}\big)
\begin{pmatrix}
 1 & 0\\
 0 & t
\end{pmatrix} ,
\]
where
\[
L(x) = P(x)\big(F\big(x+\tfrac{1}{2}\big)-tF\big(x-\tfrac{1}{2}\big)\big)\in \CN[x].
\]
It follows immediately that $Y_n$ satisfies an analogous equation
\[
Y_n\big(x+\tfrac{1}{2}\big)=A_n(x)Y_n\big(x-\tfrac{1}{2}\big)
\begin{pmatrix}
 1 & 0\\
 0 & t
\end{pmatrix} ,
\]
where
\[
A_n(x)=
\begin{pmatrix} p_n\big(x+\frac{1}{2}\big) & -q_n\big(x+\frac{1}{2}\big)\\[1ex]
 p_{n-1}\big(x+\frac{1}{2}\big) & -q_{n-1}\big(x+\frac{1}{2}\big)\end{pmatrix}
\begin{pmatrix}
 1 & t^{-1}\frac{L(x)}{P(x)}\\[1ex]
 0 & t^{-1}
\end{pmatrix}
\begin{pmatrix} p_n\big(x-\tfrac{1}{2}\big) & -q_n\big(x-\tfrac{1}{2}\big)\\[1ex]
 p_{n-1}\big(x-\tfrac{1}{2}\big) & -q_{n-1}\big(x-\tfrac{1}{2}\big)\end{pmatrix}^{-1}.
\]
Since $\det(Y_n)=N_{n-1}$, $\det(Y_0)=1$, we can use the standard formula
for the inverse of a $2\times 2$ matrix to rewrite this as
\[
A_n(x)
=
N_{n-1}^{-1}
\begin{pmatrix} p_n\big(x+\frac{1}{2}\big) & -q_n\big(x+\frac{1}{2}\big)\\[1ex]
 p_{n-1}\big(x+\frac{1}{2}\big) & -q_{n-1}\big(x+\frac{1}{2}\big)\end{pmatrix}
\begin{pmatrix}
 1 & t^{-1}\frac{L(x)}{P(x)}\\[1ex]
 0 & t^{-1}
\end{pmatrix}
\begin{pmatrix} -q_{n-1}\big(x-\tfrac{1}{2}\big) & q_n\big(x-\tfrac{1}{2}\big)\\[1ex]
 -p_{n-1}\big(x-\tfrac{1}{2}\big) & p_n\big(x-\tfrac{1}{2}\big)\end{pmatrix} .
\]
It follows immediately that $P(x)A_n(x)$ has polynomial coefficients. We~can also compute the asymptotic behavior of $A_n(x)$ using the expression
\[
A_n(x)
=
Y_n\big(x+\tfrac{1}{2}\big)
\begin{pmatrix}
 1 & 0\\
 0 & t^{-1}
\end{pmatrix}
Y_n\big(x-\tfrac{1}{2}\big)^{-1}
\]
to conclude that
\begin{alignat*}{3}
& A_n(x)_{11} = 1+\tfrac{n}{x}+O\big(\tfrac{1}{x^2}\big),\qquad&&
A_n(x)_{12} = -\tfrac{(1-t^{-1})a_n}{x} + O\big(\tfrac{1}{x^2}\big),&\\
&A_n(x)_{21} = \tfrac{1-t^{-1}}{x}+O\big(\tfrac{1}{x^2}\big),\qquad&&
A_n(x)_{22} = t^{-1}(1-\tfrac{n}{x})+O\big(\tfrac{1}{x^2}\big),&
\end{alignat*}
which when $t=1$ refines to
\begin{alignat*}{3}
&A_n(x)_{11} = 1+\tfrac{n}{x}+O\big(\tfrac{1}{x^2}\big),\qquad&&
A_n(x)_{12} = -\tfrac{(2n+1)a_n}{x^2} + O\big(\tfrac{1}{x^3}\big),&\\
&A_n(x)_{21} = \tfrac{2n-1}{x^2}+O\big(\tfrac{1}{x^3}\big),\qquad&&
A_n(x)_{22} = 1-\tfrac{n}{x}+O\big(\tfrac{1}{x^2}\big).&
\end{alignat*}

Restricting to the first column of $Y_n(x)$ gives the following.

\begin{Proposition}
 The orthogonal polynomials satisfy the difference equation
 \[
 \begin{pmatrix}
 p_n\big(x+\frac{1}{2}\big)\\[1ex]
 p_{n-1}\big(x+\frac{1}{2}\big)
 \end{pmatrix}
 = A_n(x)
 \begin{pmatrix}
 p_n\big(x-\tfrac{1}{2}\big)\\[1ex]
 p_{n-1}\big(x-\tfrac{1}{2}\big)
 \end{pmatrix} .
 \]
\end{Proposition}

Note that it is not the mere existence of a difference equation with
rational coefficients that is significant (indeed, any pair of polynomials
satisfies such an equation!), rather it is the fact that (a) the poles are
bounded independently of $n$, and (b) so is the asymptotic behavior at
infinity.

If we consider (for $t\ne 1$) the family of matrices satisfying the above
conditions; that is, $PA_n$~is polynomial, $\det(A_n)=t^{-1}$, and
\begin{gather*}
A_n(x)_{11} = 1+\tfrac{n}{x}+O\big(\tfrac{1}{x^2}\big),\qquad
A_n(x)_{12} = O\big(\tfrac{1}{x}\big),\\
A_n(x)_{21} = \tfrac{1-t^{-1}}{x} + O\big(\tfrac{1}{x^2}\big),\qquad
A_n(x)_{22} = t^{-1}\big(1-\tfrac{n}{x}\big)+O\big(\tfrac{1}{x^2}\big),
\end{gather*}
we find that the family is classified by a {\em rational} moduli space. To
be precise, let $f(x):=$ \mbox{$\big(1-t^{-1}\big)^{-1}P(x)A_n(x)_{21}$}, and let $g(x)\in
\C[x]/(f(x))$ be the reduction of $P(x)A_n(x)_{11}$ modu\-lo~$f(x)$. Then
$f$ and $g$ both vary over affine spaces of dimension $\deg(q)-1$, and
generically determine $A_n$. Indeed, $A_n(x)_{21}$ is clearly determined
by $f$, and since $A_n(x)_{11}P(x)$ is specified by the asymptotics up to
an additive polynomial of degree $\deg(P)-2$, it is determined by $f$ and
$g$. For~generic $f$, $g$, this also determines $A_n(x)_{22}$, since the
determinant condition implies that for any root $\alpha$ of $f$,
$A_n(\alpha)_{11}A_n(\alpha)_{22}=t^{-1}$. Moreover, this constraint forces
$P(x)^2 \big(A_n(x)_{11}A_n(x)_{22}-t^{-1}\big)$ to be a multiple of $f(x)$, and
thus the unique value of~$A_n(x)_{12}$ compatible with the determinant
condition gives a matrix satisfying the desired conditions.

Moreover, given such a matrix, the three-term recurrence for orthogonal
polynomials tells us that the corresponding $A_{n+1}$ is the unique matrix
satisfying {\em its} asymptotic conditions and having the form
\[
A_{n+1}(x)=
\begin{pmatrix}
 x+\frac{1}{2}-b_n & -a_n\\
 1 & 0
\end{pmatrix}
A_n(x)
\begin{pmatrix}
 x-\frac{1}{2}-b_n & -a_n\\
 1 & 0
\end{pmatrix}^{-1}.
\]
It is straightforward to see that $a_n$, $b_n$ are determined by the
leading terms in the asymptotics of $A_n(x)_{12}$, and thus in particular
are rational functions of the parameters. We~thus find that the map from
the space of matrices $A_n$ to the space of matrices $A_{n+1}$ is a
rational map, and by considering the inverse process, is in fact
birational, corresponding to a sequence $F_n$ of~bira\-tional automorphisms
of $\mathbb A^{2\deg(P)-2}$. Note that the equation $A_0$, though not of the
standard form, is still enough to determine $A_1$, and thus gives
(rationally) a $\mathbb P^{\deg(P)-1}$ worth of initial conditions corresponding
to orthogonal polynomials. (There is a $\deg(P)$-dimensional space of
valid functions $F$, but rescaling $F$ merely rescales the trace, and thus
does not affect the orthogonal polynomials.)

\begin{Example}
As an example, consider the case $P(x)=x^2$, corresponding, e.g., to
\[
{\rm w}(y)=\frac{{\rm e}^{2\pi cy}}{\cosh^2\pi y},
\]
with $c\in (0,1)$. In this case,
$\deg(P)=2$, so we get a $2$-dimensional family of linear equations, and
thus a second-order nonlinear recurrence, with a 1-parameter family of
initial conditions corresponding to orthogonal polynomials. Since the
monic polynomial $f$ is linear, we may use its root as one parameter $f_n$,
and $g_n=A_n(f_n)_{11}$ as the other parameter. We~thus find that
\begin{gather}
A_n(x)=
\begin{pmatrix}
\big(1-\frac{f_n}{x}\big)\big(1+\frac{f_n+n}{x}\big)+\frac{f_n^2 g_n}{x^2}
&-a_n\frac{1-t^{-1}}{x}\big(1-\frac{f_{n+1}}{x}\big)
\\
\frac{1-t^{-1}}{x}\big(1-\frac{f_n}{x}\big)
&t^{-1}\big(\big(1-\frac{f_n}{x}\big)\big(1+\frac{f_n-n}{x}\big)+\frac{f_n^2}{g_n x^2}\big)
\end{pmatrix} ,
\end{gather}
where
\begin{gather}
a_n = \frac{t}{(t-1)^2}\frac{n^2g_n-f_n^2(g_n-1)^2}{g_n}
\end{gather}
and $f_n$, $g_n$ are determined from the recurrence
\begin{gather}
f_{n+1} =\frac{f_n(f_n(g_n-1)-ng_n)(f_n(g_n-1)-n)}{n^2g_n-f_n^2(g_n-1)^2},
\\
g_{n+1} = \frac{(f_n(g_n-1)-ng_n)^2}{t g_n (f_n(g_n-1)-n)^2}.
\end{gather}
The three-term recurrence for the orthogonal polynomials is then
\[
p_{n+1}(x)=(x-b_n)p_n(x)-a_np_{n-1}(x),
\]
where $a_n$ is as above and
\[
b_n = -f_{n+1}-\frac{(t+1)\big(n+\frac{1}{2}\big)}{t-1}.
\]
The initial condition is given by
\[
f_0 = b_0 + \frac{t+1}{2(t-1)},\qquad
g_0 = 1.
\]
(Note that the resulting $A_0$ is not actually correct, but this induces
the correct values for~$f_1$,~$g_1$, noting that the recurrence simplifies
for $n=0$ to $f_1=-f_0$, $g_1=1/tg_0$.) It follows from the general theory
of isomonodromy deformations~\cite{Ra} that this recurrence is a discrete
Painlev\'e equation (This will also be shown by direct computation in forthcoming work by N.~Witte.). We~also note that the recurrence satisfies a sort of
time-reversal symmetry: there is a natural isomorphism between the space of
equations for~$t$,~$n$ and the space for~$t^{-1}$,~$-n$, coming (up to a diagonal
change of basis) from the duality $A\mapsto \big(A^{\rm T}\big)^{-1}$, and this symmetry
preserves the recurrence. (This follows from the fact that if two
equations are related by the three-term recurrence, then so are their
duals, albeit in the other order.)
\end{Example}

\begin{Remark}
 The fact that $A_n(x)_{12}$ has a nice expression in terms of $a_n$ and
 $f_{n+1}$ follows more generally from the fact (via the three-term
 recurrence) that
 \[
 A_n(x)_{12} = -a_n A_{n+1}(x)_{21}.
 \]
 One similarly has
 \[
 A_n(x)_{22} = A_{n+1}(x)_{11} - \big(x+\tfrac{1}{2}-b_n\big) A_{n+1}(x)_{21},
 \]
 so that in general $f_{n+1}(x)\propto P(x)A_n(x)_{12}$ and $g_{n+1}(x) =
 P(x)A_n(x)_{22}\bmod f_{n+1}(x)$.
 In particular, applying this to $n=0$ tells us that the orthogonal polynomial case corresponds to the initial condition $f_1(x)\propto L(x)$, $g_1(x)=t^{-1}P(x)\bmod L(x)$.
\end{Remark}

The above construction fails for $t=1$, because the constraint on the
asymptotics of the off-diagonal coefficients of~$A_n$ is stricter in that
case:
\begin{gather*}
A_n(x)_{21}= \tfrac{2n-1}{x^2} + O\big(\tfrac{1}{x^3}\big),
\\
A_n(x)_{12} = O\big(\tfrac{1}{x^2}\big).
\end{gather*}
The moduli space is still rational, although the arguments is somewhat
subtler. We~can still parametrize it by $f_n(x):=P(x)A_n(x)_{12}$ and
$g_n(x):=P(x)A_n(x)_{11}\bmod f_n(x)$ as above, which is certainly enough
to determine $P(x)A_n(x)_{22}$ modulo $f_n(x)$. This still leaves two
degrees of~freedom in the diagonal coefficients, but
$\det(A_n(x))+O\big(\tfrac{1}{x^4}\big)$ depends only on the diagonal coefficients and is
linear in the remaining degrees of freedom, so we can solve for those.
Once again, having determined the coefficients on and below the diagonal,
the $21$ coefficient follows from the determinant, and can be seen to have
the correct poles and asymptotics. Note that now the dimension of the
moduli space is~$2\deg(q)-4$; that the dimension is even in both cases
follows from the existence of a canonical symplectic structure on such
moduli spaces, see~\cite{Ra}.

There is a similar reduction in the number of parameters when the trace is
even (forcing $t=(-1)^n$ and $P(x)=(-1)^n P(-x)$). The key observation in that case
is that
\[
Y_n(-x) = (-1)^n \begin{pmatrix} 1 & 0 \\ 0 & -1\end{pmatrix}Y_n(x)
 \begin{pmatrix} 1 & 0 \\ 0 & -1\end{pmatrix}
\]
implying that $A_n$ satisfies the symmetry
\[
A_n(-x)
=
\begin{pmatrix} 1 & 0 \\ 0 & -1\end{pmatrix}
A_n(x)^{-1}
\begin{pmatrix} 1 & 0 \\ 0 & -1\end{pmatrix} .
\]
Since $A_n$ is $2\times 2$ and has determinant $t^{-1}=(-1)^n$, this actually
imposes linear constraints on~the coefficients of $A_n$:
\begin{alignat*}{3}
&A_n(-x)_{11} = (-1)^n A_n(x)_{22},\qquad&&
A_n(-x)_{12} = (-1)^n A_n(x)_{12},&\\
&A_n(-x)_{21} = (-1)^n A_n(x)_{21},\qquad&&
A_n(-x)_{22} = (-1)^n A_n(x)_{11}.&
\end{alignat*}
In particular, $A_n(x)_{21}$ has only about half the degrees of freedom one
would otherwise expect, and for any root of that polynomial,
$A_n(\alpha)_{11}A_n(-\alpha)_{11}=1$, again halving the degrees of freedom (and
preserving rationality).

\begin{Example}
 Consider the case $P(x)=x^3+\beta^2x$ with $t=-1$ and even trace $\big($e.g., for
 $\beta=0$, the weight function ${\rm w}(y)=\frac{1}{\cosh^3 \pi y}\big)$. Then $A_n(x)_{21}$
 has the form $\frac{2(x^2-f_n)}{x^3+\beta^2x}$, and $A_n\big(\sqrt{f_n}\big)_{11}$ is of~norm~1, which can be parametrized in the form
 \[
 A_n\big(\sqrt{f_n}\big)_{11} = \frac{g_n+\sqrt{f_n}}{g_n-\sqrt{f_n}}.
 \]
 Applying this to both square roots gives two linear conditions on
 $A_n(x)_{11}$, which suffices to determine it, with $A_n(x)_{22}$
 following by symmetry and $A_n(x)_{12}$ from the remaining determinant
 conditions. We~thus obtain
 \[
 A_n(x)
 =
 \begin{pmatrix}
 \displaystyle 1 \!+\! \frac{n(x^2\!-\!f_n)}{x(x^2\!+\!\beta^2)}
 \!+\! \frac{2f_n(f_n\!+\!\beta^2)(g_n\!+\!x)}{(g_n^2\!-\!f_n)x(x^2\!+\!\beta^2)}
 &\displaystyle
 -2 a_n \frac{x^2-f_{n+1}}{x(x^2+\beta^2)}
 \\[2ex]
 \displaystyle 2 \frac{x^2-f_n}{x(x^2+\beta^2)}
 &\displaystyle
 \!-\!1 \!+\! \frac{n(x^2\!-\!f_n)}{x(x^2\!+\!\beta^2)} \!+\! \frac{2f_n(f_n\!+\!\beta^2)(g_n\!-\!x)}{(g_n^2\!-\!f_n)x(x^2\!+\!\beta^2)}
 \end{pmatrix} ,
 \]
 where
 \[
 a_n = -\frac{n^2}{4}+\frac{f_n(f_n+\beta^2)}{g_n^2-f_n}
 \]
 and $f_n$, $g_n$ are determined by the recurrence
 \begin{gather*}
 g_{n+1} = -\frac{n}{2} - \frac{2g_na_n}{ng_n-2f_n},
 \\
 f_{n+1} = -\frac{(ng_n-2f_n)^2 g_{n+1}^2}{4f_na_n},
 \end{gather*}
 with initial condition
 $f_1 = -\beta^2-\frac{1}{4}-a_1$, $g_1 = 0$.
\end{Example}

\begin{Remark}
 One can perform a similar calculation for the case $P(x)=x^4-e_1x^2+e_2$
 with even trace; again, one obtains a second-order nonlinear recurrence,
 but the result is significantly more complicated, even for $e_1=e_2=0$.
\end{Remark}

In each case, when the moduli space is $0$-dimensional, so that the
conditions uniquely determine the equation, we get an explicit formula for
$A_n$. This, of course, is precisely the case that the orthogonal
polynomial is classical.

\subsection*{Acknowledgements} The work of P.E.\ was partially supported by the NSF grant DMS-1502244. P.E.\ is grateful to Anton Kapustin for introducing him to the topic of this paper, and to Chris Beem, Mykola Dedushenko and Leonardo Rastelli for useful discussions. E.R.\ would like to thank Nicholas Witte for pointing out the reference~\cite{Magnus}.

\pdfbookmark[1]{References}{ref}
\LastPageEnding

\end{document}